\documentclass[]{aspm}

\ifdefined\pdfpagewidth
  \AtBeginDocument{%
    \pdfpagewidth=210mm
    \pdfpageheight=297mm
  }%
\fi

\articleinfo{}{}{}

\usepackage{verbatim}
\usepackage{amssymb}
\usepackage{amsbsy}
\usepackage{amscd}
\usepackage{amsmath}
\usepackage{amsthm}
\usepackage[mathscr]{eucal}
\usepackage{xcolor}

\newtheorem{theorem}{Theorem}
\numberwithin{defn}{section}

\newcommand{\R}{\mathbb{R}}
\newcommand{\dd}{\,\mathrm{d}}
\newcommand{\db}{{\mathrm d}}
\newcommand{\Land}{\operatorname{Land}^n(\R^d)}
\newcommand{\Landz}{\operatorname{Land}^n_0(\R^d)}
\newcommand{\Sreg}{\mathcal{S}_{\mathrm{reg}}}
\newcommand{\Sregone}{\mathcal{S}_{\mathrm{reg}, 1}}

\newcommand{\sym}{\operatorname{sym}}
\newcommand{\diver}{\operatorname{div}}

\newcommand{\norm}[1]{\left\lVert #1\right\rVert}

\theoremstyle{plain}
\newtheorem{proposition}[theorem]{Proposition}
\newtheorem{corollary}[theorem]{Corollary}
\newtheorem{lemma}[theorem]{Lemma}

\theoremstyle{definition}
\newtheorem{definition}[theorem]{Definition}

\newtheorem{assumption}[theorem]{Assumption}

\title[Stochastic completeness for landmark shape spaces]{Stochastic Completeness for Landmark Shape Spaces with Induced Metrics}

\author[K. Habermann]{Karen Habermann}
\author[M. Nishino]{Mao Nishino}
\author[S. Sommer]{Stefan Sommer}

\address{Department of Statistics, University of Warwick, Coventry, CV4 7AL, United Kingdom}

\address{Department of Mathematics, Florida State University, Tallahassee, FL 32306, United States}

\address{Østervoldgade 3, DK-1350, Copenhagen, Denmark}

\email{sommer@di.ku.dk}

\rcvdate{}
\rvsdate{}

\subjclass[2010]{}


\begin{document}

\begin{abstract}
    Landmark shape spaces arise from configurations of $n$ distinct landmarks with rigid motions factored out and scale normalized.  When equipped with sufficiently regular metrics inherited from right-invariant metrics on diffeomorphism groups, the shape spaces inherit geodesic completeness from geodesics of diffeomorphisms. We now answer the corresponding stochastic question if the Riemannian Brownian motion on landmark shape spaces exists for all time and hence rule out initially distinct landmarks colliding when following a Brownian flow. We show that with general classes of metrics, including degenerate metrics with infinitesimal rigid motions in their null space, landmark shape spaces are stochastically complete, thus making the use of Riemannian Brownian motion for modelling shape stochasticity in applied fields well-founded.
\end{abstract}

\maketitle

\section{Introduction}
\label{sec:intro}
Kendall's shape spaces \cite{kendallShapeManifoldsProcrustean1984} of configurations of $n$ landmarks (points) in $\R^d$ with rigid motions factored out and scale normalized form one of the most widely used constructions in shape analysis, morphology and medical image analysis. The metric on the shape space is here inherited from the Euclidean metric on $\R^{nd}$ and thus of $L^2$-type with no coupling between the landmarks and no guarantee for initially distinct landmarks to avoid colliding along geodesic evolutions of the configurations. A separate approach to shape analysis is equipping the shape space with a metric inherited from a right-invariant metric on diffeomorphisms of the shape domain. This construction is behind the Large deformation diffeomorphic metric mapping (LDDMM) approach, see~\cite{younesShapesDiffeomorphisms2010} for further details. While LDDMM shape spaces are traditionally considered separate from Kendall's shape spaces, the work \cite{joshi2026landmark} constructs a unification of the two approaches. In this setting, with sufficiently regular metrics on the diffeomorphism group, the shape space arising when factoring out rigid motions as well as normalizing scale and inheriting a right-invariant metric on diffeomorphisms, resulting in a stratified space, is generally geodesically complete in the sense of geodesics existing for infinite time and not exiting the shape space.

The Riemannian Brownian motion on landmark manifolds with induced metrics has been used extensively in shape analysis for modelling stochastic perturbations of shapes, taking the place of the Euclidean Brownian motion as a default null-model of stochasticity used in applied fields, e.g. in phylogenetics. For such constructions to work, it is crucial that the process is well-defined and does not explode in finite time. Precise conditions on the metric for this to be the case have recently been derived in \cite{habermann2026stochasticcompletenesslandmarkspace}, after a first preliminary result in this direction has been obtained in~\cite{habermannLongtimeExistenceBrownian2024}. However, the analogous situation when factoring out rigid motions and thus arriving at landmark shape spaces has so far not been treated.

In this paper, we show that for general classes of scalar metrics and for the conditionally positive metrics of \cite{joshi2026landmark}, the landmark shape spaces are indeed stochastically complete, precisely stated in Theorem~\ref{thm:stoch-complete} in Section~\ref{sec:stoch-completeness-proof}. We arrive at this by an application of Masamune's theorem in ~\cite{masamune2005analysis} that, under specific control of the metric, asserts that the process will stay on the regular stratum of the shape space for all time almost surely. To establish this, we need to consider the capacity of the lower-dimensional strata and use a certain volume growth control for the metric.

\subsection{Paper outline}
In Section~\ref{sec:background}, we survey the necessary background on landmark configuration spaces, particularly when quotienting out rigid motion and normalizing for scale as in Kendall's shape spaces but keeping the metric inherited from an acting diffeomorphism group. We also describe the metrics and kernel co-metrics arising from right-invariant metrics on diffeomorphisms. In Section~\ref{sec:main-assumptions-lemmas}, we state the main assumptions allowing the completeness proof and prove some common results needed in the rest of the paper. Section~\ref{sec:stoch-completeness-proof} presents the main stochastic completeness proof. In Section~\ref{sec:metrics-assumptions}, we show that for general classes of scalar metrics and for the conditionally positive landmark shape space matrix-valued kernels of \cite{joshi2026landmark}, the required assumptions are indeed fulfilled. We close in Section~\ref{sec:scale} by specifically discussing the stochastic completeness for the scale normalized configuration model.

\section{Background}
\label{sec:background}
We here outline the needed background on shape spaces, descending metrics, and stochastic completeness from Masamune's criterion.

Throughout the paper, constants may depend on the fixed parameters $n$ and $d$ and the fixed metric data. Any additional dependencies and any independence needed for limiting arguments are specified locally.

\subsection{Geometry of landmark shape spaces}
Let
\begin{equation*}
    \Land
    =
    \{(q_1,\dots,q_n)\in(\mathbb R^d)^n:
    q_i\neq q_j\text{ for all }i\neq j\}
\end{equation*}
be the landmark space of \(n\) distinct ordered landmarks in \(\mathbb R^d\).
The special Euclidean group \(SE(d)\) acts elementwise on \(\Land\). Since the
translation subgroup acts freely, we may identify the quotient by translations
with the centered landmark space
\begin{equation*}
    \Landz
    =
    \left\{
    (q_1,\dots,q_n)\in\Land:
    \sum_{i=1}^n q_i=0
    \right\}.
\end{equation*}
Thus, we have
\begin{equation*}
    \Land/SE(d)
    \cong
    \Landz/SO(d).
\end{equation*}
The remaining action of the special orthogonal group $SO(d)$ on $\Landz$ is given by
\begin{equation*}
    R\cdot(q_1,\dots,q_n)
    =
    (Rq_1,\dots,Rq_n),
    \qquad R\in SO(d).
\end{equation*}
This action is smooth and proper, since \(SO(d)\) is compact, but it is not
free in general. Hence, the quotient \(\Landz/SO(d)\) is a stratified space, and we characterize the stratification by the rank of the configuration. Throughout this paper, we work on the full-rank locus and call it the regular stratum:
\begin{equation*}
    \left\{
    q\in\Landz:
    \operatorname{rank}[q_1,\dots,q_n] = d
    \right\}/SO(d).
\end{equation*}
We denote this full-rank regular stratum by $\Sreg$.
\subsubsection{Tangent and cotangent spaces}
Let us first consider the tangent and cotangent spaces of \(\Landz\). The tangent space \(T_q\Landz\) at a configuration \(q\in\Landz\) is given by
\begin{equation*}
    T_q\Landz
    =
    \left\{
    v=(v_1,\dots,v_n)\in(\mathbb R^d)^n:
    \sum_{i=1}^n v_i=0
    \right\},
\end{equation*}
and the cotangent space \(T_q^*\Landz\) is given by
\begin{equation*}
    T_q^*\Landz
    =
    \left\{
    \alpha=(\alpha_1,\dots,\alpha_n)\in(\mathbb R^d)^n:
    \sum_{i=1}^n \alpha_i=0
    \right\}.
\end{equation*}
The natural pairing between \(T_q\Landz\) and \(T_q^*\Landz\) is given by
\begin{equation*}
    \langle \alpha, v\rangle
    =
    \sum_{i=1}^n \langle \alpha_i, v_i\rangle,
\end{equation*}
where \(\langle\cdot,\cdot\rangle\) on the right-hand side denotes the Euclidean inner product on \(\mathbb R^d\).

We will now consider the tangent and cotangent spaces for the quotient \(\Landz/SO(d)\), where we will focus on the regular stratum \(\Sreg\). The tangent space \(T_q\Sreg\) at a configuration \(q\in\Sreg\) is given by the quotient
\begin{equation*}
    T_q\Sreg
    \cong
    T_q\Landz/\mathfrak{so}(d)\cdot q,
\end{equation*}
where $\mathfrak{so}(d) = \{A\in\R^{d\times d}:A^\top=-A\}$ is the Lie algebra of $SO(d)$ and $\mathfrak{so}(d)\cdot q$ is the orbit of the infinitesimal action of $\mathfrak{so}(d)$ on $q$. The cotangent space \(T_q^*\Sreg\) is given by the annihilator of \(\mathfrak{so}(d)\cdot q\) in \(T_q^*\Landz\) through linear algebra as
\begin{equation*}
    \begin{aligned}
    T_q^*\Sreg &\cong (\mathfrak{so}(d)\cdot q)^\circ \\
    &= \left\{
    \alpha\in T_q^*\Landz:
    \langle \alpha, p(q)\rangle = 0
    \text{ for all }p\in\mathfrak{so}(d)
    \right\}.
    \end{aligned}
\end{equation*}
Both the tangent and cotangent spaces of the regular stratum \(\Sreg\) have the following representation in terms of the tangent and cotangent spaces of \(\Landz\):
\begin{align*}
    T_q\Sreg &\cong \left\{v\in T_q\Landz : \sum_{i=1}^{n}q_iv_i^\top\text{ is symmetric}\right\}, \\
    T_q^*\Sreg &\cong \left\{\alpha\in T_q^*\Landz : \sum_{i=1}^{n}q_i\alpha_i^\top\text{ is symmetric}\right\}.
\end{align*}

\subsubsection{The cometric}
We will consider a symmetric kernel of the form $K\colon\mathbb{R}^{d}\times \mathbb{R}^d\to\mathbb{R}^{d\times d}$ and define the cometric \(g^{-1}\) on \(\Landz\) by
\begin{equation*}
    g^{-1}_q(\alpha,\beta)
    =
    \sum_{i,j=1}^n \alpha_i^\top K(q_i,q_j)\beta_j.
\end{equation*}
This form of cometric appears when a right-invariant metric on the diffeomorphism group induces a metric on a shape space through the left-action of the diffeomorphisms, see~\cite{younesShapesDiffeomorphisms2010}.
We then define the cometric \(g^{-1}\) on the regular stratum \(\Sreg\) of the quotient \(\Landz/SO(d)\) by the usual Riemannian submersion construction. Using the representation of the tangent and cotangent spaces of \(\Sreg\) in terms of the tangent and cotangent spaces of \(\Landz\), this means that we define the cometric \(g^{-1}\) on \(\Sreg\) by restricting the cometric \(g^{-1}\) on \(\Landz\) to the cotangent space \(T_q^*\Sreg\) of the regular stratum \(\Sreg\).

We will consider two cases of kernels $K$.
\paragraph{\bf Positive definite translationally and rotationally invariant kernels.}
Following \cite{habermann2026stochasticcompletenesslandmarkspace}, we consider positive definite kernels \(K\) of the form
\begin{equation*}
    K(q_i, q_j) = k(\|q_i - q_j\|) I_d,
\end{equation*}
where $\|\cdot \|$ is the Euclidean norm on \(\mathbb{R}^d\), \(I_d\) is the \(d\times d\) identity matrix, and \(k\colon[0,\infty)\to\mathbb{R}\) is guaranteed to be a positive and strictly decreasing function. Such kernels are translationally and rotationally invariant.

An important class of examples comes from Sobolev operators, which are the differential operators of the form \((\mathrm{Id} - \sigma^2 \Delta)^s\) for \(s >  d/2, \sigma>0\), as discussed in \cite{habermann2026stochasticcompletenesslandmarkspace}. The corresponding kernels \(k_{\rm Matern}\), so-called Mat\'ern kernels, are the Green's functions of these operators. As we will show in Proposition \ref{prop:cpd_lower_bound_matern}, the Mat\'ern kernel associated with $(\mathrm{Id} - \sigma^2 \Delta)^{s+1}$ bounds from below the CPD cometric that we subsequently introduce, and thus provides a shortcut for comparison properties needed for the stochastic completeness.

\paragraph{\bf Conditionally positive definite kernels.} 
A class of kernels arising from a Sobolev-type operator called the screened elasticity operator are introduced in \cite{joshi2026landmark}. These are adapted to the landmark shape space setting in having infinitesimal rigid motions in their null-space, thus they are not positive definite on the full landmark tangent space. We first start with the diffeomorphism group $\mathrm{Diff}_{\mathcal{A}}(\mathbb{R}^d)$ on the ambient space $\mathbb{R}^d$ with an appropriate decay condition $\mathcal{A}$. The corresponding tangent space at a diffeomorphism $\phi$ is a space of smooth vector fields on $\mathbb{R}^d$ composed with $\phi$, again with some decay condition, which we denote by $\mathfrak{X}_{\mathcal{A}}$.  We then define the screened elasticity operator $L_s u$ for $s>d/2$ and $u\in \mathfrak{X}_{\mathcal{A}}$ by
\begin{equation*}
    L_s u = -(\mathrm{Id} - \sigma^2 \Delta)^{s} \circ \mathrm{div} \circ \mathrm{sym}\nabla u , \quad (\mathrm{sym}\nabla u)_{ij} = \partial_i u_j + \partial_j u_i,
\end{equation*}
where $\sigma>0$ is a fixed parameter and $\mathrm{div}$ is the matrix divergence $(\mathrm{div} A)_j = \sum_{i=1}^d \partial_i A_{ij}$ for a matrix field $A\colon\mathbb{R}^d\to \mathbb{R}^{d\times d}$. 

The Green's function $k_s\colon \mathbb{R}^d \to \mathbb{R}^{d\times d}$ of $L_s$ defines on the cotangent space of $\Landz/\mathrm{SO}(d)$ the cometric
\[
g_q^{-1}(\alpha,\beta)
=
\sum_{i,j=1}^n
\left\langle
\alpha_i,k_s(q_i-q_j)\beta_j
\right\rangle.
\]

\subsection{Masamune's criterion for stochastic completeness}
The main ingredient in our proof of stochastic completeness is a criterion due to Masamune \cite{masamune2005analysis}, which we state below. Let $(M,g)$ be a connected Riemannian manifold of class $C^{1,1}$ without boundary, and let $\overline{M}$ be its completion with respect to the Riemannian distance $d_g$ defined by the Riemannian metric $g$. We note that Masamune defines the completion with respect to the so-called \emph{intrinsic distance}, but in a smooth manifold this is equal to the Riemannian distance, see~\cite[Remark 1]{masamune2005analysis}.

We denote by $\partial M = \overline{M}\setminus M$ the metric boundary of $M$ and by $W_0^{1,2}(M)$ the completion of functions in $C^{1,1}$ with compact support in $M$ with respect to the norm
\begin{equation*}
    \|f\|_{1,2} = \|f\|_{L^2(M,\mu_g)} + \|\nabla f\|_{L^2(M,\mu_g)},
\end{equation*}
where $\mu_g$ is the Riemannian volume measure on $M$. We further need the following definition.
\begin{definition}[{Capacity, \cite{masamune2005analysis}}]
    Let $\Sigma \subset \overline{M}$ be a Borel set. Denote by $\mathcal{O}$ the family of open sets $O$ of $\overline{M}$ such that $\Sigma \subset O$. Let $L(O)$ be the set of functions $f\in W_{0}^{1,2}(M)$ such that $0\leq f\leq 1$ and $f|_{O\cap M} = 1$ a.e. The capacity of $\Sigma$ is defined by
    \begin{equation*}
        \mathrm{Cap}(\Sigma) = \inf_{O\in \mathcal{O}} \inf_{f\in L(O)} \|f\|_{1,2}.
    \end{equation*}
    We say that $\Sigma$ is almost polar if $\mathrm{Cap}(\Sigma) = 0$.
\end{definition}
We will now show an auxiliary lemma used in the calculation of capacity. 
\begin{lemma}
    \label{lem:subadditivity}
    For any sequence of Borel sets $\Sigma_1, \Sigma_2, \cdots \subset \overline{M}$ with $\mathrm{Cap}(\Sigma_i) = 0$, we have 
    \begin{equation*}
        \mathrm{Cap}\left(\bigcup_{i=1}^{\infty}\Sigma_i\right) = 0.
    \end{equation*}
\end{lemma}
\begin{proof}
    Let $\Sigma = \bigcup_{i=1}^{\infty}\Sigma_i$. Since each $\Sigma_i$ has capacity zero, there exists a sequence of open sets $\Sigma_i\subset O_{n,i}$ and functions $u_{n,i}\in W_{0}^{1,2}(M)$ such that $0\leq u_{n,i}\leq 1$ and $u_{n,i}|_{O_{n,i}\cap M} = 1$ so that $\|u_{n,i}\|_{1,2} \leq \frac{1}{2^{i+n}}$. We need to construct a sequence of open sets $O_n$ with $\Sigma \subset O_n$ and functions $u_n$ so that $u_n \in W_{0}^{1,2}(M)$, $0\leq u_{n}\leq 1$, $u_{n}|_{O_n \cap M} = 1$ and $\|u_{n}\|_{1,2} \to 0$ as $n\to \infty$. 

    The obvious choices are $O_n = \bigcup_{i=1}^{\infty}O_{n,i}$ and
    \begin{equation*}
        u_n = \min\left\{1, \sum_{i=1}^{\infty}u_{n, i} \right\},
    \end{equation*}
    but we need to show that $u_n$ indeed satisfies the required condition. The only nontrivial condition is that $u_n\in W_{0}^{1,2}(M)$, which we will show in detail here. We first define
    \begin{equation*}
        v_{n,N} = \min\left\{1, \sum_{i=1}^{N}u_{n, i} \right\}.
    \end{equation*}
    We note that $u_{n,i} \in W_0^{1,2}(M)$ for all $i$, so the finite sum of them is also in $W_0^{1,2}(M)$. Since $v_{n,N}$ is a normal contraction of an element in $W_0^{1,2}(M)$ in the sense of Fukushima \cite[p. 5]{fukushima1980dirichlet}, and $W_0^{1,2}(M)$ is the domain of the Riemannian Dirichlet form, we have that $v_{n,N}\in W_{0}^{1,2}(M)$.

    Now, $\|v_{n,N}\|_{1,2}$ is bounded across $N$ because
    \begin{equation*}
        \|v_{n,N}\|_{1,2} \leq \sum_{i=1}^{N}\|u_{n,i}\|_{1,2} \leq \sum_{i=1}^{N} \frac{1}{2^{i+n}}\leq \sum_{i=1}^{\infty} \frac{1}{2^{i+n}} = \frac{1}{2^n}.
    \end{equation*}
    Therefore, using the fact that $W_{0}^{1,2}(M)$ is Hilbert (to be rigorous, it has a norm equivalent to a Hilbert norm), $v_{n,N}$ has a subsequence $v_{n,N_k}$ that converges to $v_n\in W_0^{1,2}(M)$ weakly. We will now show that $v_n = u_n$. Since $W_{0}^{1,2}(M)$ embeds continuously into $L^2$, we have $v_{n,N_k} \to v_{n}$ weakly in $L^2$. We show that $v_{n,N}$ strongly converges to $u_n$ in $L^2$, and hence $v_n=u_n$ by the uniqueness of limits. 

    We have that $v_{n,N} \to u_n$ pointwise, so we need to justify the dominated convergence theorem to show that  $\|v_{n,N} - u_n\|_{L^2} \to 0$. In fact, since $x\mapsto \min\{1, x\}$ is $1$-Lipschitz,
    \begin{equation*}
        |v_{n,N} - u_n|^2 \leq \left|\sum_{i=1}^{N}u_{n, i} - \sum_{i=1}^{\infty}u_{n, i}\right|^2 \leq \left| \sum_{i > N} u_{n,i} \right|^2 \leq \left| \sum_{i=1}^{\infty} u_{n,i} \right|^2,
    \end{equation*}
    and $\left| \sum_{i=1}^{\infty} u_{n,i} \right|^2$ is integrable because the sum is $L^2$ due to
    \begin{equation*}
        \left\|\sum_{i=1}^{\infty} u_{n,i}\right\|_{L^2} \leq \sum_{i=1}^{\infty }\|u_{n,i}\|_{L^2} \leq \sum_{i=1}^{\infty }\|u_{n,i}\|_{1,2} \leq \sum_{i=1}^{\infty}\frac{1}{2^{i+n}} <\infty.
    \end{equation*}
    Therefore, dominated convergence shows that $\|v_{n,N}-u_n\|_{L^2} \to 0 $ as $N\to \infty$, and thus $v_{n} = u_n\in W_{0}^{1,2}(M)$, as required.
\end{proof}

We now state the part of Masamune's theorem we need.
\begin{theorem}[Theorem 1(iii) in \cite{masamune2005analysis}]
    \label{thm:masamune}
    Let $(M,g)$ be a connected $C^{1,1}$-Riemannian manifold without boundary. If $\partial M$ is almost polar and the volumes $v(r)$ of the metric balls of radius $r$ centered at a fixed point $x\in M$ satisfy
    \begin{equation*}
        \int^\infty \frac{r}{\log v(r)} \dd r = \infty, \quad \text{ or } \int^\infty \frac{r}{v(r)} \dd r = \infty,
    \end{equation*}
    then $(M,g)$ is stochastically complete. Here, the lower limit of the integral is any positive number.
\end{theorem}
By this theorem, to prove that $\Sreg$ is stochastically complete, we need to show the following.
\begin{enumerate}
    \item The metric boundary $\partial \Sreg$ is almost polar.
    \item The volume growth of the metric balls in $\Sreg$ satisfies the integral condition $\int^{\infty} \frac{r}{\log v(r)} \dd r = \infty$ or $\int^{\infty} \frac{r}{v(r)} \dd r = \infty$.
\end{enumerate}
These results follow from four main assumptions on the given metric and the dimensions of the ambient space, which will be detailed in the subsequent sections.

We note that the theorem applies to connected manifolds, so we will show here briefly that $\Sreg$ is in fact connected under some assumptions.

\begin{lemma}
    $\Sreg$ is connected under the assumption that $d\geq 2, n \geq d+2$.
\end{lemma}
\begin{proof}
We will show that the preimage of the $\Sreg$ under the quotient map is connected, which implies that $\Sreg$ itself is connected. The preimage can be seen as an open subset of $\mathbb{R}^{d\times (n-1)} $ obtained after removing the collision subset $C = \{q \in \mathbb{R}^{d\times (n-1)} : q_i = q_j \text{ for some } i \neq j\}$ and the rank $r$ stratum for $r\leq d-1$ that we denote by $\mathcal{S}_{r}$. We calculate the codimension of these removed subsets and use \cite[Lemma 4.5.9]{Wall2016DifferentialTopology} to find a path connecting any two points that avoids the removed subsets.

The collision subset is a union of subspaces of codimension $d$ where two fixed landmarks are equal, and $\mathcal{S}_r$ has codimension $(d-r)(n-1-r)$. We note that $(d-r)(n-1-r) - (n-d) = (n-r)(d-r-1) \geq 0$, so under our assumptions that $d\geq 2$ and $n \geq d+2$, these codimensions are all at least $2$. We now apply \cite[Lemma 4.5.9]{Wall2016DifferentialTopology} with $V=\mathbb{R}, M=\mathbb{R}^{d\times (n-1)}$ and $A_{\alpha}$ the removed subsets above. The cited result implies that the set of smooth maps from $\mathbb{R}$ to $\mathbb{R}^{d\times (n-1)}$ avoiding the removed subsets is residual in the space of smooth maps, i.e., a countable intersection of dense open sets. Since the space of smooth maps is Baire, i.e., a residual set is a dense set \cite[p.~104]{Wall2016DifferentialTopology}, maps avoiding these subsets are dense. 

Now, between two points $p,q$ that belong to the preimage of $\Sreg$, take the straight line $\gamma(t)=(1-t)p + tq ,t\in \mathbb{R}$. The evaluation map $\mathrm{ev}_t:\gamma \mapsto \gamma(t)$ at $t=0,1$ is continuous under the $C^{\infty}$ topology, so the set of smooth curves $\gamma$ with $d(\gamma(0),p)$ and $d(\gamma(1),q)$ sufficiently small is open in the $C^{\infty}$ topology, so the density argument above shows that there is an intersection with this open set, i.e., there exists a smooth curve connecting $p'$ and $q'$, each arbitrarily close to $p$ and $q$ respectively, that avoids the removed subsets. Since the preimage is open, small enough balls around $p,q$ are inside the preimage, so we can connect $p'$ to $p$, and $q'$ to $q$ by a straight line to complete the path connecting $p$ and $q$ within the preimage. This shows that the preimage is connected, and hence $\Sreg$ is connected.
\end{proof}

\section{Main assumptions and common lemmas}
\label{sec:main-assumptions-lemmas}
In this section, we introduce the assumptions on the metric that enable the completeness proof and we prove common results that will be used throughout the paper. We prove in Section~\ref{sec:metrics-assumptions} that our main assumption, Assumption~\ref{assumption:main_metric}, is satisfied in the scalar kernel case for a wide range of scalar kernels as well as for the conditionally positive matrix-valued kernel $k_s$.

We start by defining quantities describing the two possible failure modes of a landmark configuration: The explosion of the configuration, where landmarks move infinitely far apart, and the collision of landmarks, where two or more landmarks coincide. The explosion of a configuration can be avoided by controlling the following quantity:
\begin{equation*}
    r_0([q]) = \left(\sum_{i=1}^{n}\|q_i\|^2\right)^{1/2}.
\end{equation*}
On the other hand, the collision can be controlled by the following quantity:
\begin{equation*}
    \delta([q]) = \frac{1}{2}\min_{1\leq i<j\leq n} \|q_i - q_j\|.
\end{equation*}
The following assumptions imply that away from these failure modes, the metric is comparable to the quotient Euclidean metric, and that the speed of collision between landmarks is controlled in a specific manner.
\begin{assumption}
    \label{assumption:main_metric}
    Suppose there are constants $B_1,B_2>0$, $B_3\geq 0$, and a threshold $r_*>0$ such that, for every $[q]\in \Sreg$ and every quotient covector $\alpha$, we have the following.
    \begin{enumerate}
        \item (\textbf{Dimension.}) We assume $d \geq 2$.
        \item (\textbf{Number of landmarks.}) We assume $n> d+2$, that is, we have sufficiently many landmarks to avoid the degenerate case where low-rank configurations is invisible to the diffusion process (i.e., capacity zero).
        \item (\textbf{Uniform comparison under $r_0$-$\delta$ control.}) For every quotient covector $\alpha$, \begin{equation*}
            B_1 (1 \land \delta([q]))^{B_3} h_{[q]}^{-1}(\alpha, \alpha) \leq g_{[q]}^{-1}(\alpha, \alpha) \leq B_2(1 \lor r_0([q]))^2 h_{[q]}^{-1}(\alpha, \alpha),
        \end{equation*}
        where $h$ is the Riemannian metric on $\Sreg$ induced by the Euclidean metric on the landmark space, $a\land b = \min(a,b)$, and $a\lor b = \max(a,b)$.
        \item (\textbf{Collision speed control.}) For $\rho_{ij} = \|q_i-q_j\|$,  if we have $0 < \rho_{ij} <r_*$, then \begin{equation*}
            |\nabla_g \rho_{ij}|_{g} \leq B_2\rho_{ij} \sqrt{\log\left(\frac{er_*}{\rho_{ij}}\right)}.
        \end{equation*}
    \end{enumerate}
\end{assumption}
Assumption~\ref{assumption:main_metric} will be our standing assumption throughout.
Point (3) of the assumption allows us to perform calculation in Euclidean metrics where classical tools are available. On the other hand, Point (4) can be used to put a lower bound on $\rho_{ij}$, and thus $\delta$. This can be used with (3) to establish comparability between the given metric and the quotient Euclidean metric. Our first result in this direction is the following lemma, which shows that we can control $r_0$ and $\delta$ within metric balls.
\begin{lemma}
    \label{lemma:bounds_on_metric_balls}
    Fix $[q_0]\in \Sreg$ and let $B_g([q_0],R)$ denote the metric ball of radius $R$ centered at $[q_0]$. There is a constant $C_1\geq 1$, depending additionally on $[q_0]$ but independent of $R$ and $[q]$, such that, for every $R\geq 1$ and every $[q]\in B_g([q_0],R)$, we have
    \begin{equation*}
        r_0([q]) \leq C_1e^{C_1R} \quad \text{and} \quad \delta([q]) \geq C_1^{-1}e^{-C_1R^2}.
    \end{equation*}
\end{lemma}
\begin{proof}
    For $f=r_0$ and $f =\rho_{ij}$, we use the cometric bound by integrating $f$ on a curve. In particular, we use the Riemannian Cauchy-Schwarz inequality
    \begin{equation}
        \label{eqn:riemannian_cauchy_schwartz}
        \left|\frac{\db}{\db t}f(\gamma(t))\right|^2 \leq g^{-1}(\db f, \db f)\|\dot{\gamma}(t)\|_g^2.
    \end{equation}
    To make the right hand side easier, we would want an upper bound on $g^{-1}(\db f, \db f)$ along the curve $\gamma(t)$. For this, we instead consider the transformation $F = \phi(f)$ for some suitable function $\phi$. By the chain rule, we have 
    \begin{equation*}
        g^{-1}(\db F,\db F) \leq (\phi'(f))^2 g^{-1}(\db f, \db f).
    \end{equation*}
    Therefore, we want $(\phi'(f))^2 g^{-1}(\db f, \db f)$ to be bounded. For $f=r_0$, we have
    \begin{align}
    \label{eqn:cometric_bound_r0}
    g^{-1}(\db r_0,\db r_0) &\leq B_2(1 \lor r_0)^2 h^{-1}(\db r_0,\db r_0) \\ \nonumber &= B_2(1 \lor r_0)^2 |\nabla r_0|_h^2 = B_2(1 \lor r_0)^2.
    \end{align}
    So, we want $(\phi'(r_0))^2$ to be of the order of $(1 \lor r_0)^{-2}$ to make the product bounded. The function $1 \lor r_0$ tends to $1$ as $r_0 \to 0$ and is of the order of $r_0$ as $r_0 \to \infty$, which has the smooth replacement $1+r_0$. So we can take $\phi'(r_0) = \frac{1}{1+r_0}$, meaning $\phi(r_0) = \log{(1+r_0)}$. With this choice, $(\phi'(r_0))^2 g^{-1}(\db r_0, \db r_0)$ is bounded since $(1\lor r_0)/(1+r_0) \leq 1$. Thus, integrating $\frac{d}{dt}\phi(r_0(\gamma(t)))$ along a curve from $[q_0]$ to $[q]$ of length less than $1+R$ shows that, with the combination of the Riemannian Cauchy-Schwarz inequality~\eqref{eqn:riemannian_cauchy_schwartz} and the cometric bound \eqref{eqn:cometric_bound_r0},
    \begin{equation*}
        \log(1+r_0([q])) - \log(1+r_0([q_0])) \leq \sqrt{B_2}(R+1),
    \end{equation*}
    which implies
    \begin{equation*}
        r_0([q]) \leq (1+r_0([q_0]))e^{\sqrt{B_2}(R+1)} - 1.
    \end{equation*}
    We note that the base point $q_0$ is fixed here and in the downstream use of the lemma, so the factor $(1+r_0([q_0]))$ can be seen as a constant.

    Now for $f=\rho_{ij}$, we use a similar argument. By the collision speed control assumption from Assumption \ref{assumption:main_metric}, we have $g^{-1}(\db \rho_{ij},\db \rho_{ij}) \leq B_2^2\rho_{ij}^2 \log\left(\frac{er_*}{\rho_{ij}}\right)$ whenever $0< \rho_{ij}<r_*$. Therefore, $\phi'(\rho_{ij})$ must be of order $\frac{1}{\rho_{ij}\sqrt{\log\left(\frac{er_*}{\rho_{ij}}\right)}}$ to make the product bounded. This function has an elementary antiderivative of $-2\sqrt{\log\left(\frac{er_*}{\rho_{ij}}\right)}$, so we should take
    \begin{equation*}
        \phi(\rho_{ij}) = \begin{cases}
        \sqrt{\log\left(\frac{er_*}{\rho_{ij}}\right)}, & 0< \rho_{ij}<r_*, \\
        1, & \rho_{ij} \geq r_*.
        \end{cases}
    \end{equation*}
    Now we integrate $\phi(\rho_{ij})$ along a curve from $[q_0]$ to $[q]$ of length less than $1+R$ to obtain an estimate similar to the one for $r_0$:
\begin{equation*}
    \phi(\rho_{ij}([q])) - \phi(\rho_{ij}([q_0])) \leq \frac{B_2}{2}(R+1),
\end{equation*}
which implies
\begin{equation*}
    \rho_{ij}([q]) \geq er_* \exp\left(-(\phi(\rho_{ij}([q_0])) + \frac{B_2}{2}(R+1))^2\right).
\end{equation*}
Set $C_2=\max_{i<j}\phi(\rho_{ij}([q_0]))$. For $R\geq1$, we can replace $R+1$ by $R+R = 2R$ in the exponent, so halving and taking the minimum of $\rho_{ij}$ gives
\begin{equation*}
    \delta([q]) \geq \frac{er_*}{2}e^{-2C_2^2}e^{-2B_2^2R^2}.
\end{equation*}
Both claimed bounds now follow by choosing once
\begin{equation*}
    C_1=\max\left\{1,\sqrt{B_2},(1+r_0([q_0]))e^{\sqrt{B_2}},
    2B_2^2,\frac{2e^{2C_2^2}}{er_*}\right\}.
\end{equation*}
\end{proof}
We will now use the estimate to show a comparison result on the $g$-metric balls.
\begin{corollary}
\label{cor:comparison_g_metric_balls}
Fix $[q_0]\in\Sreg$. There are constants $D_1,D_2\geq1$, depending additionally on $[q_0]$ but independent of $R$ and $[q]$, such that, for every $R\geq1$, on $B_g([q_0],R)$,
\begin{equation}\label{eq:comparison}
    D_1^{-1}e^{-D_1R^2}h^{-1}\leq g^{-1} \leq D_1e^{D_1R}h^{-1}.
\end{equation}
On the same ball, the volume densities satisfy
\begin{equation*}
    \frac{\db\mu_g}{\db\mu_h} \leq D_2e^{D_2R^2},
\end{equation*}
where $\mu_g$ and $\mu_h$ are the volume measure associated with the metrics $g$ and $h$, respectively.
\end{corollary}
\begin{proof}
    We will use Lemma~\ref{lemma:bounds_on_metric_balls}. Since $C_1\geq1$, Assumption~\ref{assumption:main_metric} gives
    \begin{equation*}
        g^{-1} \leq B_2(1\lor r_0)^2h^{-1}
        \leq B_2C_1^2e^{2C_1R}h^{-1}.
    \end{equation*} 
    Similarly, $C_1^{-1}e^{-C_1R^2}\leq1$, so the lower comparison gives
    \begin{equation*}
        g^{-1} \geq B_1(1\land\delta)^{B_3}h^{-1}
        \geq B_1C_1^{-B_3}e^{-B_3C_1R^2}h^{-1}.
    \end{equation*}
    Thus \eqref{eq:comparison} holds with the fixed choice
    \begin{equation*}
        D_1=\max\{1,B_2C_1^2,2C_1,B_1^{-1}C_1^{B_3},B_3C_1\}.
    \end{equation*}

    We will next show the claimed comparison result for the volume elements. From the lower bound in \eqref{eq:comparison}, we obtain
    \begin{equation*}
        g \leq D_1e^{D_1R^2}h.
    \end{equation*}
    Since the volume form is the square root of the determinant of the metric, we deduce
    \begin{equation*}
       \db\mu_g\leq D_1^{D/2}e^{(D/2)D_1R^2}\db\mu_h,
    \end{equation*}
    where $D=d(n-1)-d(d-1)/2$ is the dimension of $\Sreg$. The density estimate then follows with the fixed choice $D_2=\max\{1,D_1^{D/2},(D/2)D_1\}$.
\end{proof}

We will now prove a characterization of the metric boundary, where we show that the boundary is contained in the set of points where the smallest singular value of the landmark configuration matrix is zero. We denote by $\sigma_{k}([q])$ the $k$-th singular value of the matrix $[q_1,\dots,q_n]$. We remark that all of these quantities are well-defined on the quotient $\Sreg$ since they are invariant under the action of $SO(d)$. With this in mind, we prove the following.
\begin{proposition}
    \label{prop:boundary_characterization}
The functions $r_0, \delta$, and $\sigma_{d}$ on $\Sreg$ all extend continuously to its completion. Moreover, any point $[q]\in \partial \Sreg$ satisfies $r_0([q])< \infty$, $\delta([q])>0$ and $\sigma_{d}([q]) = 0$, where $r_0, \delta$ and $\sigma_d$ denote the continuous extensions.
\end{proposition}
\begin{proof}
    We first show that $r_0, \delta$ and $\sigma_d$ extend continuously with respect to the metric of interest. For this, we show that all these functions are uniformly continuous on any bounded set, which implies that the functions are Cauchy continuous, i.e., Cauchy sequences map to Cauchy sequences. We can then use the fact that Cauchy continuous functions extend continuously to the completion \cite[Theorem 19.27]{Schechter1997}.

    Any bounded set on $\Sreg$ is contained in a metric ball $B_g([q_0],R)$ for some $R\geq1$, so it suffices to show uniform continuity on such metric balls. The crucial property of $r_0,\delta$ and $\sigma_d$ is that they are all Lipschitz continuous with respect to the Euclidean distance on the landmark configuration matrices, which implies uniform continuity. Keep the constant $D_1$ from Corollary~\ref{cor:comparison_g_metric_balls} for the fixed center $[q_0]$. On $B_g([q_0],R)$, the corollary gives $h\leq D_1e^{D_1R}g$. We will now convert this into an inequality of distances.
    
    For any points in the metric ball $[x], [y] \in B_g([q_0], R)$ close enough, e.g., $d_g([x], [y]) < 1$, take $\epsilon>0$ so that $[x], [y]\in B_g([q_0], R-\epsilon)$, e.g, $\epsilon = (R - \max\{d_g([q_0], [x]),  d_g([q_0], [y]) \})/2$. Now, we can take a curve $\gamma_\epsilon$ connecting $[x]$ and $[y]$ with length at most $1 + \epsilon$. This curve stays in $B_g([q_0], R + 1)$. In fact, any points on this curve are at most a distance $(1+\epsilon)/2$ from either $[x]$ or $[y]$, so the triangle inequality shows $d_g(\gamma(t), [q_0]) \leq R - \epsilon + \frac{1+\epsilon}{2} = R - \frac{\epsilon}{2} + \frac{1}{2} < R+1$. Therefore, the curve $\gamma_\epsilon$ remains entirely within the metric ball $B_g([q_0], 1+R)$, and we can use the comparison to show that 
    \begin{equation*}
         d_h([x],[y])\leq L_h(\gamma_\epsilon)
         \leq\sqrt{D_1}e^{D_1(R+1)/2}L_g(\gamma_\epsilon).
    \end{equation*}
    Define $E_1(L)=\sqrt{D_1}e^{D_1L/2}$ for $L\geq1$. This constant $E_1$ depends on the actual comparison radius $L$ and the fixed center, but is independent of the endpoints and the curves. Taking the infimum over all such curves $\gamma_\epsilon$ connecting $[x]$ and $[y]$, we obtain
    \begin{equation*}
         d_h([x],[y])\leq E_1(R+1)d_g([x],[y]).
    \end{equation*}
    Since $r_0, \delta$ and $\sigma_d$ are Lipschitz continuous with respect to $d_h$, that translates to $d_g$ too as long as the distance $d_g([x], [y])$ is sufficiently small. This is enough to prove the uniform continuity, and hence the Cauchy continuity, of these functions on any bounded set.

    We will now show the characterization of the boundary. For any $[q]\in \partial \Sreg$, we have a Cauchy sequence $[q_n]\in \Sreg$ converging to $[q]$ in the completion. Since any Cauchy sequence is bounded, $[q_n]$ lies in some metric ball $B_g([q_0], R)$ for sufficiently large $R$. By Lemma \ref{lemma:bounds_on_metric_balls}, $r_0$ and $\delta$ are both bounded from above and below on this metric ball, respectively, so the continuity shows that $r_0([q]), \delta([q])$ are finite and positive, respectively. 

    For the singular value, for the sake of contradiction let us assume $\sigma_d([q]) > 0$. We first realize the abstract completion point $[q]\in \partial \Sreg$ as a concrete configuration. Take a representation $q_n$ of the Cauchy sequence $[q_n]$ in $\Sreg$ converging to $[q]$. Since $r_0$ is bounded uniformly, the sequence $q_n$ is also bounded in the Euclidean (Frobenius) norm. By the compactness of the closed and bounded sets in the Euclidean space, there exists a subsequence $q_{n_k}$ that converges to some matrix $\bar{q}$ in the Euclidean norm. We will show that $[\bar{q}] = [q]$ in the completion. The continuity argument gives $\delta(\bar{q})>0$ and $\sigma_d(\bar{q})>0$, so $\bar{q}$ is collision-free and of full rank, i.e., $[\bar{q}]\in\Sreg$. Set $E_2=\delta(\bar{q})/2$ and take a small enough Euclidean ball around $\bar{q}$ so that $\delta>E_2$ and $\sigma_d>0$ throughout the ball. Assumption~\ref{assumption:main_metric} gives $g\leq B_1^{-1}(1\land E_2)^{-B_3}h$ there. Define $E_3=B_1^{-1/2}(1\land E_2)^{-B_3/2}$, where constants depend on the limiting configuration but not on the subsequence index. A straight line path $\gamma$ between $q_{n_k}$ and $\bar{q}$ stays in the ball for large $k$, so
    \begin{equation*}
        d_g([q_{n_k}],[\bar{q}])\leq L_g(\gamma)
        \leq E_3L_h(\gamma)\leq E_3\|q_{n_k}-\bar{q}\|_F.
    \end{equation*}
    Taking the limit as $k\to \infty$, we obtain
    \begin{equation*}
        d_g([q_{n_k}], [\bar{q}]) \to 0,
    \end{equation*}
    which shows that $[\bar{q}] = [q]$ in the completion, as desired. This contradicts the assumption $[q]\in \partial \Sreg$, hence we must have $\sigma_d([q]) = 0$.    
\end{proof}

\section{Stochastic completeness}
\label{sec:stoch-completeness-proof}
In this section, we will combine the results from the previous sections to prove the main theorem of stochastic completeness of the regular stratum $\Sreg$ of the quotient $\Landz/SO(d)$ with respect to the cometrics $g^{-1}$ satisfying Assumption~\ref{assumption:main_metric}.
\begin{theorem}
    Under Assumption~\ref{assumption:main_metric}, the manifold $(\Sreg,g)$ is stochastically complete.
    \label{thm:stoch-complete}
\end{theorem}
The proof as presented in Section~\ref{sec:actual-proof} below relies on the following two facts, as required by Masamune's theorem.
\begin{enumerate}
    \item The metric boundary $\partial \Sreg$ is almost polar.
    \item The volume growth of the metric balls in $\Sreg$ satisfies the integral condition in Masamune's theorem.
\end{enumerate}
We establish these results below.

\subsection{Almost polar boundary}
To show that a set $S\subset \bar{\mathcal{S}}_{\mathrm{reg}}$ is almost polar, we need to construct a family of functions $u_\epsilon \in W_0^{1,2}(\Sreg)$ that are $1$ on a neighborhood of $S$ and have vanishing $W^{1,2}$ norm as $\epsilon \to 0$. Therefore, we need a tool to estimate the $W^{1,2}$ norm of a function, and in particular an integral on $\Sreg$. 
As discussed in Proposition~\ref{prop:boundary_characterization}, the boundary $\partial \Sreg$ is contained in the set of points where the smallest singular value vanishes. Therefore, $u_{\epsilon}$ for $\partial \Sreg$ should be a function of the smallest singular value, vanishing as the smallest singular value grows away from zero. Thus, it is natural to use the SVD coordinate to calculate the integral. We will refer to the approach by Díaz-García \cite{DiazGarcia2013More} for the SVD calculation.

We first consider the pre-quotient space $\Landz$ and SVD
coordinates on it. Let $q\in\Landz$, viewed as a $d\times n$
matrix. Choose an orthonormal basis of the hyperplane
\[
    \mathbf{1}^{\perp}
    = \left\{x\in\mathbb{R}^{n}:
        \sum_{i=1}^{n}x_i=0\right\},
\]
and let $U\in\mathbb{R}^{n\times(n-1)}$ have these basis
vectors as its columns. Here, $\mathbf{1}\in \mathbb{R}^n$ is the vector with all entries equal to $1$. Since $q\mathbf{1}=0$, we may
identify $q$ with the $d\times(n-1)$ matrix $qU$, again
denoted by $q$. We then consider its SVD. Díaz-García considers the so-called \emph{thin} SVD: for $k=\mathrm{rank}(q)$,  $q = H_1 D P^{\top}$ where $H_1\in V_{k, d}, D = \mathrm{diag}(\sigma_1(q), \dots, \sigma_k(q))$ and $P\in V_{k, n - 1}$. We then have the following result on the change of variables formula.
\begin{theorem}[Theorem 3.1 in \cite{DiazGarcia2013More}]
    \label{thm:volume_density_prequotient}
    Let $X$ be a $d\times (n-1)$ matrix of rank $k$ and $X = H_1 D P^{\top}$ be its SVD. Assume that the singular values of $X$ are distinct. Then the Euclidean volume density $\db X$ can be expressed as
    \begin{equation*}
        \db X = 2^{-k} \prod_{i=1}^{k} \sigma_i^{n+d-1-2k} \prod_{i<j} (\sigma_i^2 - \sigma_j^2) (\db D) \wedge (H_1^\top \db H_1) \wedge (P^\top \db P),
    \end{equation*}
    where $(\db D) = \bigwedge_{i=1}^{k} \db\sigma_i$, $(H_1^\top \db H_1)$ is the volume density on the Stiefel manifold $V_{k, d}$ and $P^\top \db P$ is the volume density on the Stiefel manifold $V_{k, n - 1}$.
\end{theorem}
We remark that although the above theorem assumes that the singular values are distinct, we can still use the result to calculate the integral on the whole space $\Landz$ since the set of matrices with repeated singular values has measure zero. In fact, the set of $q$ with repeated roots is a zero set of discriminants of the characteristic polynomial and discriminants are polynomials in the entries of $q$. Therefore, the set of $q$ with repeated singular values is a zero set of a polynomial and has measure zero.

We will now derive the density formula on the quotient. 
\begin{theorem}
    \label{thm:volume_density_quotient}
    The volume density $\db \mu_h$ on the quotient $\Sreg$ can be expressed in terms of the SVD coordinates as
    \begin{equation*}
        \db\mu_h = \prod_{i=1}^{d} \sigma_i^{n-d-1} \prod_{i<j} \frac{(\sigma_i^2 - \sigma_j^2)}{(\sigma_i^2 + \sigma_j^2)^{1/2}} (\db D) \wedge (P^\top \db P).
    \end{equation*}
\end{theorem}
\begin{proof}
    Since the quotient map $\pi$ from full-rank locus of $\Landz$ to the quotient $\Sreg$ is a Riemannian submersion, orthogonal decomposition shows that we can decompose the volume density $\db X$ on the pre-quotient space $\Landz$ as
    \begin{equation*}
        \db X = \mathrm{vol}_{\mathrm{vert}} \otimes \pi^*\db\mu_h,
    \end{equation*}
    where $\mathrm{vol}_{\mathrm{vert}}$ is the volume density on the vertical bundle of the Riemannian submersion and $\db\mu_h$ is the Euclidean volume density on the quotient $\Sreg$. Contracting both sides with the oriented orthonormal vertical frame $V_1, \dots, V_{d(d-1)/2}$, we obtain the volume density on the quotient $\Sreg$ as
    \begin{equation*}
        \pi^* \db\mu_h = \iota_{V_1}\cdots \iota_{V_{d(d-1)/2}} \db X.
    \end{equation*}
    We would like to substitute general vertical vectors $U_1, \dots, U_{d(d-1)/2}$ instead of the orthonormal vertical frame. Using the change of basis formula for interior products, we have
    \begin{equation*}
        \iota_{V_1}\cdots \iota_{V_{d(d-1)/2}} \db X = \frac{\iota_{U_1}\cdots \iota_{U_{d(d-1)/2}} \db X}{\sqrt{\det(\langle U_i, U_j\rangle)}},
    \end{equation*}
    where $\langle \cdot, \cdot\rangle$ is the Euclidean inner product. At $q = H_1 D P^{\top}$, take $U_{ij} = H_1\frac{E_{ij}-E_{ji}}{\sqrt{2}}H_1^T q = H_1\frac{E_{ij}-E_{ji}}{\sqrt{2}} D P^{\top}$ for $1\leq i<j\leq d$ as the vertical vectors, where $E_{ij}$ is the $d\times d$ matrix with $1$ in the $(i,j)$-th entry and $0$ elsewhere. We obtain 
    \begin{align*}
        \langle U_{ij}, U_{kl}\rangle &= \frac{1}{2}\mathrm{tr}(P D(E_{ij}-E_{ji})^\top H_1^\top H_1(E_{kl}-E_{lk})D P^\top) \\ &= \frac{1}{2}\mathrm{tr}((E_{ij}-E_{ji})^\top(E_{kl}-E_{lk})D^2) \\ &= \frac{1}{2}\left\langle (E_{ij} - E_{ji}) D, (E_{kl} - E_{lk})D \right\rangle.
    \end{align*}
    Note that $E_{ij} D = \sigma_j E_{ij}$ is the matrix with $\sigma_j$ in the $(i,j)$-th entry and $0$ elsewhere, so
    \begin{equation*}
        \langle U_{ij}, U_{kl}\rangle = \frac{1}{2}\langle \sigma_jE_{ij} - \sigma_i E_{ji}, \sigma_l E_{kl} - \sigma_k E_{lk}\rangle,
    \end{equation*}
    which is equal to $\frac{\sigma_i^2 + \sigma_j^2}{2}$ if $(i,j) = (k,l)$ and $0$ otherwise. Therefore, we have
    \begin{equation*}
        \pi^* \db\mu_h = \frac{2^{d(d-1)/4}}{\prod_{i<j}(\sigma_i^2 + \sigma_j^2)^{1/2}} \iota_{U_1}\cdots \iota_{U_{d(d-1)/2}} \db X.
    \end{equation*}
    We finally calculate the interior product $\iota_{U_1}\cdots \iota_{U_{d(d-1)/2}} \db X$ using the SVD coordinate formula for $\db X$: by taking the curve given by $\gamma(t) = H_1\exp(t(E_{ij}-E_{ji})/\sqrt{2})D P^{\top}$, we have
    \begin{align*}
        (H_1^{\top}\db H_1)_{ij}(U_{kl}) &= \left(H_1^{\top} H_1 \frac{E_{kl}-E_{lk}}{\sqrt{2}}\right)_{ij} = \left(\frac{E_{kl}-E_{lk}}{\sqrt{2}}\right)_{ij} \\ &= \frac{1}{\sqrt{2}}\delta_{(i,j),(k,l)}.
    \end{align*}
    This yields
    \begin{equation*}
        \pi^* \db\mu_h = \prod_{i=1}^{d} \sigma_i^{n+d-1-2d} \prod_{i<j} \frac{(\sigma_i^2 - \sigma_j^2)}{(\sigma_i^2 + \sigma_j^2)^{1/2}} (\db D) \wedge (P^\top \db P).
    \end{equation*}
    We remark that we do not include the constant $2^{-d}$ here since Theorem \ref{thm:volume_density_prequotient} should be read as the global integration formula rather than the local differential form that is calculated here. It comes from the fact that there are $2^d$ equivalent SVD representations of the same matrix due to the sign ambiguities in the singular vectors.

    Since $\pi$ is a submersion, we can remove the pullback $\pi^*$ and obtain the volume density on the quotient $\Sreg$ as
    \begin{equation*}
        \db\mu_h = \prod_{i=1}^{d} \sigma_i^{n-d-1} \prod_{i<j} \frac{(\sigma_i^2 - \sigma_j^2)}{(\sigma_i^2 + \sigma_j^2)^{1/2}} (\db D) \wedge (P^\top \db P),
    \end{equation*}
    as claimed.
\end{proof}
We will use the above formula to calculate the following bound.
    \begin{proposition}
        \label{ref:singular_value_bound}
        For each $R>0$, there is a constant $G_1(R)>0$, depending only on $n,d$ and $R$, such that, for every $\epsilon>0$ and every nonnegative measurable function $w\colon(0,\infty)\to[0,\infty)$,
        \begin{equation}\label{eq:weighted-singular-value-bound}
            \int_{\{r_0<R,\ 0<\sigma_d<\epsilon\}}w(\sigma_d)\dd\mu_h
            \leq(n-d)G_1(R)\int_0^\epsilon w(s)s^{n-d-1}\dd s.
        \end{equation}
        In particular, taking $w\equiv1$ gives the volume bound
        \begin{equation*}
            \mu_h(\{[q]\in\Sreg:r_0([q])<R,\ 0<\sigma_d([q])<\epsilon\})
            \leq G_1(R)\epsilon^{n-d}.
        \end{equation*}
        The constant $G_1(R)$ is independent of $\epsilon$, $w$ and the metric $g$.
    \end{proposition}
    \begin{proof}
    Fix $R,\epsilon>0$ and a nonnegative measurable function $w$. An $r_0$ bound translates to a bound on the singular values since $r_0([q])^2 = \sum_{i=1}^{d} \sigma_i([q])^2$. We use ordered singular values $\sigma_1\geq\cdots\geq\sigma_d>0$. Let $\db \nu(P)$ denote the density on $V_{d,n-1}$ after absorbing any normalization or multiplicity constants and let $G_2=\int_{V_{d,n-1}}\db\nu(P)$. This is a fixed finite constant depending only on $n$ and $d$. Since $w\geq0$, the SVD formula and enlargement of the integration domain give
    \begin{align*}
        &\int_{\{r_0<R,\ 0<\sigma_d<\epsilon\}}w(\sigma_d)\dd\mu_h \\
        &\leq \int_{V_{d,n-1}}\int_{\substack{0<\sigma_d\leq\cdots\leq\sigma_1<R\\\sigma_d<\epsilon}}
        w(\sigma_d)\prod_{i=1}^{d}\sigma_i^{n-d-1}
        \prod_{i<j}\frac{\sigma_i^2-\sigma_j^2}{(\sigma_i^2+\sigma_j^2)^{1/2}}\dd D\dd\nu(P)\\
        &\leq \int_{V_{d,n-1}}\int_{\substack{0<\sigma_d\leq\cdots\leq\sigma_1<R\\\sigma_d<\epsilon}}
        w(\sigma_d)\prod_{i=1}^{d}\sigma_i^{n-d-1}
        \prod_{i<j}\sqrt{\sigma_i^2+\sigma_j^2}\dd D\dd\nu(P)\\
        &\leq \int_{V_{d,n-1}}\int_{\substack{0<\sigma_d\leq\cdots\leq\sigma_1<R\\\sigma_d<\epsilon}}
        w(\sigma_d)\sigma_d^{n-d-1}\prod_{i=1}^{d-1}R^{n-d-1}
        \prod_{i<j}\sqrt{2}R\dd D\dd\nu(P)\\
        &\leq  G_2\,2^{d(d-1)/4}R^{(d-1)(n-d)+d(d-1)/2}
        \int_0^\epsilon w(s)s^{n-d-1}\dd s.
    \end{align*}
    In the last step, we set $s=\sigma_d$, enlarge the domain of the other singular values to $(0,R)^{d-1}$, and integrate them and the angular variables. 

    With $D=d(n-1)-d(d-1)/2$, choose once
    \begin{equation}\label{eq:singular-volume-constant}
        G_1(R)=\frac{G_2\,2^{d(d-1)/4}}{n-d}R^{D-(n-d)}.
    \end{equation}
    Since $(d-1)(n-d)+d(d-1)/2=D-(n-d)$, the preceding estimate proves \eqref{eq:weighted-singular-value-bound}.

    Finally, setting $w\equiv1$, we obtain
    \begin{align*}
        \mu_h\{r_0<R,\ 0<\sigma_d<\epsilon\}
        \leq(n-d)G_1(R)\int_0^\epsilon s^{n-d-1}\dd s
        =G_1(R)\epsilon^{n-d}.
    \end{align*}
    The same volume bound holds for non-strict cutoffs by taking limits from above.
    \end{proof}

We can now prove the following theorem.
\begin{theorem}
    \label{thm:boundary-almost-polar}
$\partial \Sreg$ is almost polar.
\end{theorem}
\begin{proof}
    To make use of the comparison assumption in Assumption \ref{assumption:main_metric}, we decompose the boundary $\partial \Sreg$ into the pieces where $r_0$ and $\delta$ are controlled:
    \begin{align*}
        F_j &= \left\{[q]\in \overline{\Sreg}^{g}: r_0([q])\leq j, \delta([q]) \geq  \frac{1}{j}, \sigma_{d}([q]) = 0\right\}, \quad j\in \mathbb{N}, \\ 
        \partial \Sreg &= \bigcup_{j=1}^{\infty} F_j.
    \end{align*}
    We will show that $F_j$ has capacity zero, which implies that $\partial \Sreg$ has capacity zero by Lemma \ref{lem:subadditivity}.
    
    To show that $\mathrm{Cap}(F_j)=0$, we construct a family of functions $u_\epsilon \in W_0^{1,2}(\Sreg)$ that are $1$ on a neighborhood of $F_j$ and have vanishing $W^{1,2}$ norm as $\epsilon \to 0$: 
    \begin{equation*}
        u_{\epsilon} = \eta_j \psi_{\epsilon}(\sigma_{d}([q])),
    \end{equation*}
    where $\eta_j$ is a bump function on a neighborhood of $F_j$ and $\psi_\epsilon$ is a decreasing function on $[0,\infty)$ that vanishes as its argument grows, so it is in fact a function that is $1$ near $F_j$ and decays to $0$ as the singular value moves away from zero, i.e., away from $F_j$.

    We will describe the main steps of the proof. Our goal is to show that the integral
    \begin{equation*}
        \int_{\Sreg} \left(|u_\epsilon|^2 + |\nabla u_\epsilon|^2 \right)\db\mu_g
    \end{equation*}
    converges to zero as $\epsilon \to 0$. By definition of the capacity, this integral bounds the capacity of $F_j$ from above as follows by Cauchy-Schwarz:
    \begin{equation}
        \mathrm{Cap}(F_j) \leq \|u_{\epsilon}\|_{1,2} \leq \sqrt{2}\left( \int_{\Sreg} \left(|u_\epsilon|^2 + |\nabla u_\epsilon|^2 \right) \db\mu_g\right)^{1/2},
    \end{equation}
    so convergence to zero shows that $\mathrm{Cap}(F_j) = 0$.
    Fix $j$ and write $D=\dim\Sreg$. We will collect the coefficients in a single constant $J_1(j)$, independent of $\epsilon$ and $\tau$, to obtain the final estimate
    \begin{equation}\label{eq:capacity-energy-bound}
        \int_{\Sreg}(|u_\epsilon|^2+|\nabla u_\epsilon|^2)\dd\mu_g
        \leq J_1(j)\big((2\epsilon)^{n-d}+(2\epsilon)^{n-d-2}\big).
    \end{equation}
    
    Since $0\leq\eta_j,\psi_\epsilon\leq1$, we have $|u_\epsilon|^2\leq1$, and the main task is to control the gradient term $|\nabla u_\epsilon|^2$. We then evaluate the $\db\mu_g$ volume of the support by the SVD estimate provided earlier. The chain rule gives us
    \begin{equation*}
        \nabla u_\epsilon = \nabla \eta_j \psi_\epsilon(\sigma_d([q])) + \eta_j \psi_\epsilon'(\sigma_d([q])) \nabla \sigma_d([q]).
    \end{equation*}
    The smallest singular value is $1$-Lipschitz in the Frobenius norm, so Rademacher's theorem and the quotient construction give $|\nabla_h\sigma_d|_h\leq1$ almost everywhere. Therefore, we need to carefully construct $\eta_j$ and $\psi_\epsilon$ such that the gradient term is bounded. With this in mind, our four steps are as follows:
    \begin{enumerate}
        \item Construct a bump function $\eta_j$ on a neighborhood of $F_j$ such that $\nabla \eta_j$ is bounded.
        \item Construct a decreasing function $\psi_\epsilon$ such that the integral of $|\psi_\epsilon'|$ is bounded.
        \item Show that $u_{\epsilon}$ is in $W_0^{1,2}(\Sreg)$.
        \item Bound $|u_\epsilon|$ and $|\nabla u_\epsilon|$ to show that $\|u_{\epsilon}\|_{1,2} \to 0$.
    \end{enumerate}

    \paragraph{\bf Step 1: Construct a bump function $\eta_j$.} To construct $\eta_j$, we first construct it on $\Landz$ and then project it down to $\Sreg$. We first take a bump function $\tilde{\eta}_j$ that is $1$ on the compact set 
    \begin{equation*}
        \tilde{F}_j = \left\{q\in \Landz: r_0(q)\leq j + 1, \delta(q) \geq  \frac{1}{j + 1} \right\}
    \end{equation*}
    and is supported on the open set
    \begin{equation*}
        \tilde{U}_j = \left\{q\in \Landz: r_0(q)< j + 2, \delta(q) >  \frac{1}{j + 2} \right\}.
    \end{equation*}
    To make the function invariant with respect to the action of $SO(d)$, we average it over the group $SO(d)$:
    \begin{equation*}
        \bar{\eta}_j(q) = \int_{SO(d)} \tilde{\eta}_j(R\cdot q) \dd R,
    \end{equation*}
    where $\db R$ is the normalized Haar measure on $SO(d)$. It is straightforward to check that $\bar{\eta}_j$ is $1$ on $\tilde{F}_j$, supported on $\tilde{U}_j$, and has bounded gradient. Finally, we project it down to $\Sreg$ by defining $\eta_j([q]) = \bar{\eta}_j(q)$. By the invariance of $\bar{\eta}_j$ and the definition of Riemannian submersion, we have that $\|\nabla_h \eta_j([q])\|_{h} = \|\nabla_{\mathrm{Euc}} \bar{\eta}_j(q)\|_{\mathrm{Euc}}$, which is bounded. Here, $\nabla_h$ and $\|\cdot\|_h$ denote the gradient and norm with respect to the Riemannian metric $h$ on $\Sreg$ and similarly $\nabla_{\mathrm{Euc}}$ and $\|\cdot\|_{\mathrm{Euc}}$ denote the gradient and norm with respect to the Euclidean metric on $\Landz$.
    Choose $0\leq\tilde\eta_j\leq1$, so averaging preserves $0\leq\eta_j\leq1$, and fix
    \begin{equation*}
        J_2(j)=\max\{1,\|\nabla_h\eta_j\|_{L^\infty}\}.
    \end{equation*}
    This constant depends on the chosen cutoff but is independent of $\epsilon$ and $\tau$.
    
    \paragraph{\bf Step 2: Construct a decreasing function $\psi_\epsilon$ with a bounded derivative.} We construct $\psi_\epsilon$ by the simple linear transition
    \begin{equation*}
        \psi_{\epsilon}(s) = \begin{cases}
            1, & s \leq \epsilon, \\
            2 - \frac{s}{\epsilon}, & \epsilon < s < 2\epsilon, \\
            0, & s \geq 2\epsilon.
        \end{cases}
    \end{equation*}
    It is clear that $\psi_\epsilon$ is decreasing and that $|\psi_\epsilon'|$ is bounded by $\frac{1}{\epsilon}$.
    
    \paragraph{\bf Step 3: Show that $u_\epsilon$ is in $W_0^{1,2}(\Sreg)$.}
    It is enough to show that there exists a family $u_{\epsilon, \tau} \in W_0^{1,2}(\Sreg)$ so that $u_{\epsilon, \tau} \to u_{\epsilon}$ in $W^{1,2}$ since the $W_0^{1,2}(\Sreg)$ is defined as the completion of $C^{1,1}$-functions with compact support, so it is closed. The reason that $u_{\epsilon}$ may not have a compact support is because we do not control the lowest singular value $\sigma_d$ from below, so the support may approach the boundary, that is, points satisfying $\sigma_d=0$. To fix this, we add another cutoff function to control $\sigma_d$ from below. Therefore, we define
    \begin{equation*}
        u_{\epsilon, \tau} = \eta_j \psi_\epsilon(\sigma_d) (1 - \psi_\tau(\sigma_d)) =  (1 - \psi_\tau(\sigma_d))u_{\epsilon}.
    \end{equation*}
    The support of this function is in fact compact, since its support is closed and bounded on the pre-quotient space $\Landz$ by the bound on $r_0$, and the projection to $\Sreg$ is continuous. Moreover, the support is inside $\Sreg$, since the extra cutoff provided by $(1 - \psi_\tau(\sigma_d))$ ensures that the function vanishes near the boundary where $\sigma_d = 0$.
    
    The function $u_{\epsilon,\tau}$ is also in $W^{1,2}(\Sreg)$. In fact, each of $\eta_j$, $\psi_{\epsilon}$, $\psi_\tau$ and $\sigma_d$ is Lipschitz in Euclidean metric, and the product and composition of Lipschitz functions are also Lipschitz, so $u_{\epsilon,\tau}$ is Lipschitz in Euclidean metric, and hence bounded on the compact support. The Rademacher theorem shows that the derivative is also bounded in Euclidean metric. Now the comparison assumption in Assumption \ref{assumption:main_metric} shows that the derivative is also bounded in the metric of interest. Since the support of $u_{\epsilon, \tau}$ is compact, the volume of the support is finite, so we have that both $\|u_{\epsilon, \tau}\|_{L^2}$ and $\|\nabla u_{\epsilon, \tau}\|_{L^2}$ are finite, so $u_{\epsilon, \tau} \in W_0^{1,2}(\Sreg)$.

    Now, assuming that $\tau$ is small enough, e.g., $2\tau <\epsilon$, we have 
    \begin{equation*}
        u_{\epsilon} - u_{\epsilon,\tau} = \eta_{j}\psi_{\epsilon}(\sigma_d) \psi_{\tau}(\sigma_d) = \eta_j \psi_{\tau}(\sigma_d) = u_{\tau}.
    \end{equation*}
    Therefore, it is enough to show that $u_{\tau} \to 0$ in $W^{1,2}$ as $\tau \to 0$. This is done in the next step with $\epsilon$ in place of $\tau$. 

    \paragraph{\bf Step 4: Bound $|u_\epsilon|$ and $|\nabla u_\epsilon|$ to show that $\|u_{\epsilon}\|_{1,2} \to 0$.}
    On the support of $\eta_j$, Assumption~\ref{assumption:main_metric} gives
    \begin{equation}
        \label{eq:metric_comparison_on_support}
        g^{-1}\leq B_2(j+2)^2h^{-1},\qquad
        \db\mu_g\leq\frac{(j+2)^{B_3D/2}}{B_1^{D/2}}\dd\mu_h.
    \end{equation}
    By definition of $u_\epsilon$, we have $|u_\epsilon|\leq1$. Keeping $G_1$ from Proposition~\ref{ref:singular_value_bound} at the actual radius $j+2$, we obtain
    \begin{align*}
        \int_{\Sreg}|u_\epsilon|^2\dd\mu_g
        &\leq\frac{(j+2)^{B_3D/2}}{B_1^{D/2}}
        \mu_h\{r_0<j+2,\ 0<\sigma_d<2\epsilon\}\\
        &\leq\frac{(j+2)^{B_3D/2}}{B_1^{D/2}}G_1(j+2)(2\epsilon)^{n-d}.
    \end{align*}
    Now, we bound the gradient term. By the chain rule, we have
    \begin{equation*}
        \nabla u_\epsilon = \nabla \eta_j \psi_\epsilon(\sigma_d([q])) + \eta_j \psi_\epsilon'(\sigma_d([q])) \nabla \sigma_d([q]).
    \end{equation*}
    Using $|\nabla_h\eta_j|_h\leq J_2(j)$, $|\nabla_h\sigma_d|_h\leq1$, and the preceding cometric comparison, we have
    \begin{equation*}
        |\nabla u_\epsilon|\leq\sqrt{B_2}(j+2)J_2(j)
        \big(\psi_\epsilon(\sigma_d)+|\psi_\epsilon'(\sigma_d)|\big).
    \end{equation*}
    Since $\eta_j$ is supported where $r_0<j+2$ and $\delta>1/(j+2)$, the integral of the gradient term can be bounded as follows:
    \begin{align*}
        &\int_{\Sreg}|\nabla u_\epsilon|^2\dd\mu_g \notag\\
        &\quad\leq 2B_2(j+2)^2J_2(j)^2 \notag\\
        &\qquad{}\times\int_{\{\substack{r_0<j+2,\ \delta>1/(j+2)\\0<\sigma_d<2\epsilon}\}}
        \big(\psi_\epsilon(\sigma_d)^2+|\psi_\epsilon'(\sigma_d)|^2\big)\dd\mu_g\\
        &\quad\leq\frac{2B_2(j+2)^{2+B_3D/2}J_2(j)^2}{B_1^{D/2}} \notag\\
        &\qquad{}\times\int_{\{\substack{r_0<j+2,\ \delta>1/(j+2)\\0<\sigma_d<2\epsilon}\}}
        \big(\psi_\epsilon(\sigma_d)^2+|\psi_\epsilon'(\sigma_d)|^2\big)\dd\mu_h\\
        &\quad\leq\frac{2B_2(j+2)^{2+B_3D/2}J_2(j)^2}{B_1^{D/2}}\,(n-d)G_1(j+2) \notag\\
        &\qquad{}\times\int_0^{2\epsilon}\big(\psi_\epsilon(s)^2+|\psi_\epsilon'(s)|^2\big)s^{n-d-1}\dd s\\
        &\quad\leq\frac{2B_2(j+2)^{2+B_3D/2}J_2(j)^2G_1(j+2)}{B_1^{D/2}} \notag\\
        &\qquad{}\times\big((2\epsilon)^{n-d}+4(2\epsilon)^{n-d-2}\big).
    \end{align*}
    Here, this first inequality comes from the preceding bound \eqref{eq:metric_comparison_on_support} on $|\nabla u_\epsilon|$ and the support of $\eta_j$. The second inequality again uses the density comparison by \eqref{eq:metric_comparison_on_support}. The third inequality uses the SVD estimate by Proposition~\ref{ref:singular_value_bound} with $w(s) = \psi_{\epsilon}(s)^2 + |\psi_\epsilon'(s)|^2$. The last inequality evaluates the integrals using $|\psi_\epsilon|\leq1$ and $|\psi_\epsilon'|\leq\epsilon^{-1}$.

    Combining the $L^2$ and gradient estimates proves \eqref{eq:capacity-energy-bound} with the single fixed choice
    \begin{equation*}
        J_1(j)=\frac{(j+2)^{B_3D/2}G_1(j+2)}{B_1^{D/2}}
        \big(1+8B_2(j+2)^2J_2(j)^2\big).
    \end{equation*}
    The constants $J_1(j)$ and $J_2(j)$ may depend on the chosen bump function $\eta_j$, but are independent of $\epsilon$ and $\tau$. Hence, we have established that $\|u_\epsilon\|_{1,2}\to0$ as $\epsilon\to0$, which shows that $\mathrm{Cap}(F_j)=0$. By Lemma \ref{lem:subadditivity}, we have
    \begin{equation*}
        \mathrm{Cap}(\partial \Sreg) =0.
    \end{equation*}
    Therefore, $\partial \Sreg$ is almost polar as defined in the sense of Masamune.
\end{proof}

\subsection{Volume growth of metric balls}
We now calculate the bound for the volume of the metric balls in our setting. Our goal is to establish an upper bound so that the logarithm of the volume grows at most quadratically with the radius of the balls.
\begin{proposition}
    \label{prop:volume_growth_metric_balls}
    Fix $[q_0]\in\Sreg$. There is a constant $K_1\geq1$, depending additionally on $[q_0]$ but independent of $R$, such that, for every $R\geq1$,
    \begin{equation*}
        \mu_g(B_g([q_0],R))\leq K_1e^{K_1R^2}.
    \end{equation*}
\end{proposition}
\begin{proof}
    Keep $C_1$ from Lemma~\ref{lemma:bounds_on_metric_balls} and $D_2$ from Corollary~\ref{cor:comparison_g_metric_balls}, both for the fixed base point $[q_0]$. The density comparison gives
    \begin{align*}
        \mu_g(B_g([q_0],R))\leq D_2e^{D_2R^2}
        \int_{B_g([q_0],R)}\db\mu_h.
    \end{align*}
    The metric-ball bound gives $r_0([q])\leq C_1e^{C_1R}$ throughout the ball. Since $r_0$ bounds every singular value, Proposition~\ref{ref:singular_value_bound}, evaluated at the actual spatial cutoff $C_1e^{C_1R}$, yields
    \begin{align*}
        \int_{B_g([q_0],R)}\db\mu_h
        &\leq\mu_h\{r_0\leq C_1e^{C_1R},\ 0<\sigma_d\leq C_1e^{C_1R}\}\\
        &\leq G_1(C_1e^{C_1R})(C_1e^{C_1R})^{n-d}\\
        &=G_1(C_1)C_1^{n-d}e^{C_1DR},
    \end{align*}
    where $D=d(n-1)-d(d-1)/2$ and the equality uses the formula of $G_1(R)$ in Equation \eqref{eq:singular-volume-constant}. Combining the estimates, for $R\geq1$, we have $R\leq R^2$, so we obtain
    \begin{align*}
        \mu_g(B_g([q_0],R))
        &\leq D_2G_1(C_1)C_1^{n-d}e^{D_2R^2+C_1DR}\\
        &\leq D_2G_1(C_1)C_1^{n-d}e^{(D_2+C_1D)R^2}.
    \end{align*}
    The stated bound follows by choosing once
    \begin{equation*}
        K_1=\max\{1,D_2G_1(C_1)C_1^{n-d},D_2+C_1D\}.
    \end{equation*}
\end{proof}

\subsection{Proof of the stochastic completeness}
\label{sec:actual-proof}
We have now proven that the boundary is almost polar and that the volume growth of metric balls is not too fast. These results are enough to invoke Theorem \ref{thm:masamune} to show the stochastic completeness of the manifold.
\begin{proof}[Proof of Theorem~\ref{thm:stoch-complete}.]
    We have already shown that the boundary is almost polar, so it is enough to show that, for a fixed base point $[q_0]\in\Sreg$, the volume of the ball $v(r) = \mu_g(B_g([q_0],r))$ satisfies
    \begin{equation*}
        \int^{\infty} \frac{r}{\log{v(r)}}\dd r = \infty, \quad \text{or} \quad \int^{\infty} \frac{r}{v(r)}\dd r = \infty.
    \end{equation*}
    We consider in two cases: $\mu_g(\Sreg) < \infty$ and $\mu_g(\Sreg) = \infty$. If $\mu_g(\Sreg) < \infty$, then $v(r)$ is bounded by a constant, say $v(r) \leq C$ for all $r$. In this case,
\begin{equation*}
    \int^{\infty} \frac{r}{v(r)}\dd r \geq \int^{\infty} \frac{r}{C}\dd r = \infty.
\end{equation*}
    Hence, the integral condition is satisfied in this case. We will now consider the case $\mu_g(\Sreg) = \infty$. In this case, the volume of the metric balls grows without bound as $r \to \infty$, so $\log{v(r)}$ is positive for large enough $r$.
    By Proposition~\ref{prop:volume_growth_metric_balls}, we have $v(r) \leq K_1 e^{K_1 r^2}$ for some constant $K_1>0$ independent of $r$. Therefore, we have $\log{K_1} \leq K_1 r^2$ for $r\geq 1$, so 
    \begin{align*}
        \int^{\infty} \frac{r}{\log{v(r)}}\dd r
        &\geq \int^{\infty} \frac{r}{\log{(K_1 e^{K_1 r^2})}}\dd r\\
        &= \int^{\infty} \frac{r}{\log{K_1} + K_1 r^2}\dd r\\
        &\geq \int^{\infty} \frac{r}{K_1 r^2+ K_1 r^2}\dd r\\
        &= \frac{1}{2K_1} \int^{\infty} \frac{1}{r}\dd r\\
        &= \infty.
    \end{align*}
    Hence, by Theorem \ref{thm:masamune}, the manifold $(\Sreg,g)$ is stochastically complete.
\end{proof}

\section{Metrics satisfying our assumptions}
\label{sec:metrics-assumptions}
We now prove that Assumption~\ref{assumption:main_metric} is satisfied in general cases for metrics given by scalar kernels and for the conditionally positive definite matrix-valued kernel $k_s$.

\subsection{Positive-definite scalar kernels}
In the scalar positive definite case, we start with the following kernel assumption. The assumption was used in \cite{habermann2026stochasticcompletenesslandmarkspace} for the proof of stochastic completeness in the positive definite case.
\begin{assumption}
\label{asm:kernel hypothesis} The function $k\colon [0, \infty)\to \mathbb{R}$ is smooth on $(0,\infty)$ and $q\mapsto k(\|q\|)$ on $\R^d$ has a positive continuous Fourier transform $\widehat{k}\colon\mathbb{R}^d\to \mathbb{R}$. Moreover, there exist a threshold $\varepsilon_0>0$, constants $A_1,A_2>0$, and an exponent $p\geq 1$, such that
\begin{equation*}
    k(0) - k(s) \leq -A_1s^2 \log{s}, \quad 0<s<\varepsilon_0,
\end{equation*}
and, for all sufficiently large $\xi$,
\begin{equation}\label{eq:HS-assumptions_2}
    \widehat{k}(\xi) \geq A_2\|\xi\|^{-p}.
\end{equation}
\end{assumption}

On quotient covectors define
\begin{equation}\label{eq:scalar-cometric}
 g_k^{-1}(\alpha,\beta)
 =\sum_{i,j=1}^n k(\rho_{ij})\alpha_i^{\mathsf T}\beta_j.
\end{equation}

Our goal in this subsection is to show the following theorem.
\begin{theorem}\label{thm:scalar-metric-assumption} Assume $d\geq 2$ and $n > d+2$.
The scalar-kernel metric \eqref{eq:scalar-cometric} satisfies Assumption~\ref{assumption:main_metric} for any kernel satisfying Assumption~\ref{asm:kernel hypothesis}.
\end{theorem}

\subsubsection{Verifying the cometric comparison}

\begin{proposition}\label{prop:scalar-verification}
Fix an exponent $p_*>\max\{p,d\}$. The scalar-kernel metric \eqref{eq:scalar-cometric} satisfies the comparison part of Assumption~\ref{assumption:main_metric} with $B_3=p_*-d$. In particular, there are constants $M_1,M_2>0$, independent of the configuration and the quotient covector, such that
\begin{equation*}
    M_1(1\land\delta)^{p_*-d}h^{-1}(\alpha,\alpha)
    \leq g_k^{-1}(\alpha,\alpha)\leq M_2h^{-1}(\alpha,\alpha).
\end{equation*}
\end{proposition}

\begin{proof}
Since $k$ has a positive Fourier transform, we have $|k(r)| \leq k(0)$ for all $r\geq0$. Therefore,
\begin{align*}
 g_k^{-1}(\alpha,\alpha)
 &\leq \sum_{i,j=1}^n |k(\rho_{ij})|\|\alpha_i\| \|\alpha_j\| \leq k(0)\left(\sum_{i=1}^n\|\alpha_i\|\right)^2\\
 &\leq nk(0)\sum_{i=1}^n\|\alpha_i\|^2
 =nk(0)h^{-1}(\alpha,\alpha).
\end{align*}

The choice $p_*>p$ preserves the inequality \eqref{eq:HS-assumptions_2} for large enough $\|\xi\|$ with exponent $p_*$ and the original coefficient $A_2$. Positivity and continuity of $\widehat k$ on compact sets, together with this high-frequency lower bound, give a fixed constant $M_3>0$ such that, for all $M\geq1$,
\[
 \inf_{\|\xi\|\leq2M}\widehat k(\xi)\geq M_3M^{-p_*}.
\]
Here $M_3$ depends on the fixed kernel and the chosen exponent, but not on $M$ or the configuration. Take $M=7d/(1\land\delta)$, so $M\geq1$ and $M\geq7d/\delta$. Theorem 4.1 in \cite{habermann2026stochasticcompletenesslandmarkspace} gives
\begin{align*}
 g_k^{-1}(\alpha,\alpha)
 &\geq\frac{M_3}{2^{1+3d/2}\Gamma(1+d/2)}M^{d-p_*}h^{-1}(\alpha,\alpha)\\
 &=\frac{M_3(7d)^{d-p_*}}{2^{1+3d/2}\Gamma(1+d/2)}
 (1\land\delta)^{p_*-d}h^{-1}(\alpha,\alpha).
\end{align*}
Thus the final comparison constants can be fixed as
\begin{equation*}
    M_1=\frac{M_3(7d)^{d-p_*}}{2^{1+3d/2}\Gamma(1+d/2)},
    \qquad M_2=nk(0).
\end{equation*}
This proves the comparison part of our Assumption~\ref{assumption:main_metric} with $B_1=M_1$, $B_3=p_*-d$, and any $B_2\geq M_2$.
\end{proof}
\subsubsection{Verifying the collision speed control}
We now verify the collision speed control using the assumption that $k(0)-k(r)\leq-A_1r^2\log r$ for $0<r<\varepsilon_0$.
\begin{proposition}
    \label{prop:pd-speed-control}
    The scalar-kernel metric \eqref{eq:scalar-cometric} satisfies the speed control part of Assumption~\ref{assumption:main_metric} with $r_* = \varepsilon_0$.
\end{proposition}
\begin{proof}
    We will calculate $g^{-1}_k(\db\rho_{ij}, \db\rho_{ij})$. By direct calculation of the covectors, we have
    \begin{equation*}
        \db\rho_{ij} = (\nabla_{q_1} \rho_{ij}, \nabla_{q_2} \rho_{ij}, \dots, \nabla_{q_n} \rho_{ij}).
    \end{equation*}
    Each component is calculated as
    \begin{equation*}
        \nabla_{q_k} \rho_{ij} =
        \begin{cases}
            \frac{q_k - q_j}{\rho_{ij}}, & \text{if } k=i,\\
            \frac{q_k - q_i}{\rho_{ij}}, & \text{if } k=j,\\
            0, & \text{otherwise}.
        \end{cases}
    \end{equation*}
    Therefore, we obtain
    \begin{align*}
    g^{-1}_k(\db\rho_{ij}, \db\rho_{ij}) &= \sum_{k,l=1}^{n}k(\rho_{kl}) \nabla_{q_k} \rho_{ij}^{\top} \nabla_{q_l} \rho_{ij} \\
    &= k(\rho_{ii}) + k(\rho_{jj}) - 2 k(\rho_{ij}) \\
    &= 2(k(0) - k(\rho_{ij})).
    \end{align*}
    Set
    \begin{equation*}
        C_{\varepsilon_0}=\max\{1,\log(1/\varepsilon_0)\}.
    \end{equation*}
    For $0<r<\varepsilon_0$, we have $\log(\varepsilon_0/r)>0$, and hence
    \begin{align*}
        -\log r
        &=\log(1/\varepsilon_0) + \log(\varepsilon_0/r)\\
        &\leq C_{\varepsilon_0}\bigl(1+\log(\varepsilon_0/r)\bigr)
        =C_{\varepsilon_0}\log(e\varepsilon_0/r).
    \end{align*}
    Thus, whenever $0<\rho_{ij}<\varepsilon_0$, the kernel assumption gives
    \begin{align*}
        g^{-1}_k(\db\rho_{ij}, \db\rho_{ij})
        &\leq -2A_1\rho_{ij}^2\log\rho_{ij}\\
        &\leq 2A_1C_{\varepsilon_0}\rho_{ij}^2
        \log\left(\frac{e\varepsilon_0}{\rho_{ij}}\right).
    \end{align*}
    Taking square roots yields
    \begin{equation*}
        |\nabla_{g_k}\rho_{ij}|_{g_k}
        \leq\sqrt{2A_1C_{\varepsilon_0}}\,\rho_{ij}
        \sqrt{\log\left(\frac{e\varepsilon_0}{\rho_{ij}}\right)}.
    \end{equation*}
    Therefore, both the comparison and the collision-speed conditions of Assumption~\ref{assumption:main_metric} hold with the single fixed choices
    \begin{equation*}
        B_2=\max\{M_2,\sqrt{2A_1C_{\varepsilon_0}}\},
        \qquad M_2=nk(0),
        \qquad r_*=\varepsilon_0.
    \end{equation*}
\end{proof}

\subsection{The CPD metric}

As we recall from Section \ref{sec:background}, the conditionally positive definite metrics constructed in \cite{joshi2026landmark} are obtained through the screened elasticity operator defined by 
\begin{equation*}
    L_s u = -(\mathrm{Id}-\sigma^2 \Delta)^{s} \circ \mathrm{div} \circ \sym{\nabla u}, \quad (\sym \nabla u)_{ij} = \partial _i u_j+\partial_j u_i,
\end{equation*} 
where $s>d/2, \sigma>0$ is a fixed parameter and $\mathrm{div}$ is the matrix divergence $(\mathrm{div} A)_j = \sum_{i=1}^d \partial_i A_{ij}$ for a matrix field $A\colon\mathbb{R}^d\to \mathbb{R}^{d\times d}$. The Green's function $k_s\colon \mathbb{R}^d \to \mathbb{R}^{d\times d}$ of $L_s$ is used to define the CPD metric on the cotangent space of $\Landz$, and then on the quotient $\Landz/\mathrm{SO}(d)$ as 
\begin{equation*}
    g^{-1}_s(\alpha,\beta) = \sum_{i,j=1}^n \alpha_i^\top k_s(q_i-q_j) \beta_j.
\end{equation*}
Our goal in this section is to show the following theorem.
\begin{theorem} Assume $d\geq 2$ and $n > d+2$.
    The CPD metric satisfies Assumption \ref{assumption:main_metric} if $s > d/2$, with appropriate constants $B_1, B_2 > 0$ and $r_* > 0$.
\end{theorem}

\subsubsection{Fourier representation of cometric}

We will first write down the Fourier representation of our cometric $g_s^{-1}$. Before doing so, we will set up the convention for Fourier transforms. For $u\colon \mathbb{R}^d \to \mathbb{R}^d$, define its Fourier transform $\hat{u}\colon\mathbb{R}^d\to \mathbb{C}^d$ and its inverse Fourier transform by
\begin{equation*}
    \hat{u}(\xi) = \frac{1}{(2\pi)^{d/2}}\int_{\mathbb{R}^d} u(x) e^{-i x^\top \xi} \dd x, \quad 
    u(x) = \frac{1}{(2\pi)^{d/2}}\int_{\mathbb{R}^d} \hat{u}(\xi) e^{i x^\top \xi} \dd\xi.
\end{equation*}
In general, the definition should be understood in the sense of distributions.

We will now take the Fourier transform of both sides of the equation defining the Green's function
\begin{equation*}
    L_s k_s = \delta_0 I_d,
\end{equation*}
where $\delta_0$ is the Dirac delta centered at $x=0$ and $I_d$ is the $d\times d$ identity matrix. Calculating the Fourier transform directly, we obtain
\begin{align*}
    \widehat{\diver\circ \sym{\nabla u}}(\xi)_j
    &= \sum_{i}i\xi_i(i\xi_i \hat{u}_j(\xi) + i\xi_j \hat{u}_i(\xi)) \\
    &= -\sum_i \xi_i^2 \hat{u}_j(\xi) - \sum_i \xi_i \xi_j \hat{u}_i(\xi).
\end{align*}
Therefore, we have $\widehat{\diver\circ \sym{\nabla u}}(\xi) = -\left(\|\xi\|^2 I_d + \xi \xi^\top\right) \hat{u}(\xi)$. Thus, the Fourier transform of the Green's function satisfies
\begin{equation*}
    \widehat{L_s k_s}(\xi) = (1+\sigma^2 \|\xi\|^2)^{s} \left(\|\xi\|^2 I_d + \xi \xi^\top\right) \hat{k}_s(\xi) = \frac{1}{(2\pi)^{d/2}} I_d.
\end{equation*}
To calculate the inverse of the matrix multiplying $\hat{k_s}(\xi)$, we notice that the matrix $\|\xi\|^2 I_d + \xi \xi^{\top}$ can be decomposed as the sum of projections into orthogonal spaces as follows:
\begin{equation*}
\|\xi\|^2 I_d + \xi \xi^{\top} = \|\xi\|^2 \left(I_d - \frac{\xi \xi^{\top}}{\|\xi\|^2} + 2\frac{\xi \xi^{\top}}{\|\xi\|^2}\right) = \|\xi\|^2(P_{\perp}(\xi) + 2 P_{||}(\xi)),    
\end{equation*}
where $P_{||}(\xi)$ and $P_{\perp}(\xi)$ are the projections parallel and perpendicular to the vector $\xi$, respectively, that is,
\begin{equation*}
    P_{||}(\xi) = \frac{\xi \xi^{\top}}{\|\xi\|^2}, \quad P_{\perp}(\xi) = I_d - \frac{\xi \xi^{\top}}{\|\xi\|^2}.
\end{equation*}
Therefore, $\hat{k}_s(\xi)$ can be written as 
\begin{align*}
    \hat{k}_s(\xi) &= \frac{1}{(2\pi)^{d/2} (1+\sigma^2 \|\xi\|^2)^{s}\|\xi\|^2} \left(P_{\perp}(\xi) + \frac{1}{2} P_{||}(\xi)\right).
\end{align*}
We remark that this is a formal calculation and rigorously there is an issue at $\xi=0$. In particular, after the polar coordinate transformation, the singularity $\frac{1}{\|\xi\|^2}$, multiplied with the polar Jacobian $r^{d-1}$, turns into $r^{d-3}$, and thus $\hat{k}_s(\xi)$ is not integrable in a neighborhood of the origin for $d = 2$. Therefore, we need to regularize the function and interpret the inverse Fourier transform in the sense of tempered distributions.

In the following justification, we consider $d=2$. Moreover, denote the formal Fourier transform as $M(\xi)$. Recall that a \emph{tempered distribution} in our case is a continuous linear functional on the Schwartz space of rapidly decreasing smooth matrix-valued functions $\mathcal{S}(\mathbb{R}^2, \mathbb{C}^{2\times 2})$.

A \emph{regularization} of a function $f\colon \mathbb{R}^2 \to \mathbb{R}^{2\times 2}$ locally integrable on $\mathbb{R}^2 \setminus \{x_0\}$ for $x_0 \in \mathbb{R}^2$ is a tempered distribution $F$ such that $\langle F, \phi \rangle = \int_{\mathbb{R}^2} f(x) \phi(x) \dd x$ for all test functions $\phi \in \mathcal{S}(\mathbb{R}^2, \mathbb{R}^{2\times 2})$ for which the integral exists in the Lebesgue sense \cite[p.~54]{EstradaKanwal1994}. As Estrada and Kanwal \cite{EstradaKanwal1994} discuss, for the algebraic singularity we have, the regularization of $M(\xi)$, denoted $\hat{k}_s$, can be obtained by
\begin{align*}
    \langle \hat{k}_s, \phi \rangle &= \int_{\|\xi\|\leq 1} M(\xi) \cdot (\phi(\xi) - \phi(0)) \dd\xi + \int_{\|\xi\|> 1} M(\xi)\cdot  \phi(\xi) \dd\xi.
\end{align*}
Here, the notation $\cdot$ denotes the Frobenius inner product between matrices, i.e., $A \cdot B = \mathrm{tr}(A^{\top} B)$. We remark that the algebraic singularity will be canceled by the first order Taylor term exposed by the difference $\phi(\xi) - \phi(0)$. We will now retroactively define $k_s$ as the distributional inverse Fourier transform of $\hat{k}_s$, that is,
\begin{equation*}
    \langle k_s, \phi \rangle = \langle \hat{k}_s, \hat{\phi}(-\cdot) \rangle \quad \textrm{for all } \phi \in \mathcal{S}(\mathbb{R}^2, \mathbb{R}^{2\times 2}),
\end{equation*}
where $\hat{\phi}(-\cdot)$ denotes the Fourier transform of $\phi$ evaluated at $-\xi$. It is straightforward to show that the distribution $k_s$ defined in this way satisfies the defining equation $L_s k_s = \delta_0 I_d$ after taking the Fourier transform, which is valid because Fourier transform is bijective on the space of tempered distributions. Now, the $k_s$ defined in this way can be represented as a continuous matrix-valued function $\mathbb{R}^2 \to \mathbb{R}^{2\times 2}$, because for a smooth bump function $\chi\colon \mathbb{R}^2\to \mathbb{R}$ compactly supported around $0$, we can write 
\begin{equation*}
    \hat{k}_s = \chi \hat{k}_s + (1-\chi) \hat{k}_s.
\end{equation*}
The first term is a compactly supported distribution, and so its inverse Fourier transform has a representation as an entire function, which is $x\mapsto (2\pi)^{-1}[\langle \chi \hat{k}_s, e^{i\xi^{\top} x} E_{ij} \rangle]_{i,j=1}^d$, see~\cite[Theorem 7.1.14]{Hormander1990}. Here, $E_{ij}$ is the matrix with $1$ in the $(i,j)$-th entry and $0$ elsewhere. The second term is supported away from the singularity $0$, so it is represented as an integrable function $(1-\chi(\xi))M(\xi)$. Therefore, the inverse Fourier function of $\hat{k}_s$ can be written as the sum of an entire function and a continuous function, which shows that $k_s$ is indeed a continuous matrix-valued function of the following representation elementwise:
\begin{align}
    \label{eq:ks-continuous-representation-d2}
    k_s(x)_{ij} &= (2\pi)^{-1}\left[\langle \chi \hat{k}_s, e^{i\xi^{\top} x} E_{ij} \rangle + \int_{\mathbb{R}^2} (1-\chi(\xi)) M(\xi)_{ij} e^{i\xi^{\top} x} \dd\xi\right]\\ \nonumber
    &= (2\pi)^{-1}\left[\int_{\|\xi\|\leq 1} M(\xi)_{ij}(e^{i\xi^{\top}x} - 1)\dd\xi +  \int_{\|\xi\|> 1} M(\xi)_{ij} e^{i\xi^{\top} x} \dd\xi \right].
\end{align}

We now go back to the $d \geq 3$ case. In such a case, $M(\xi)$ is well-behaved and thus regularization is not needed to write the inverse Fourier transform directly as
\begin{equation}
    k_s(x) = \frac{1}{(2\pi)^{d/2}} \int_{\mathbb{R}^d} M(\xi) e^{i\xi^{\top} x} \dd\xi.
\end{equation}
We first derive the Fourier representation of the cometric in this case. Recall that our cometric is defined by 
\begin{equation*}
    g_s^{-1}(\alpha,\beta) = \sum_{i,j=1}^{n}\alpha_i^{\top} k_s(q_i-q_j) \beta_j.
\end{equation*}
Now, we would like to substitute the inverse Fourier transform into $k_s$. We obtain
\begin{align*}
 g_{s}^{-1}(\alpha,\beta) &= \sum_{i,j=1}^{n}\alpha_i^{\top} \frac{1}{(2\pi)^{d/2}}\int_{\mathbb{R}^d}M(\xi)e^{i\xi^{\top}(q_i-q_j)} \beta_j \dd\xi \\ &= \frac{1}{(2\pi)^{d/2}}\int_{\mathbb{R}^d}\left(\sum_{i=1}^{n}\alpha_ie^{-i\xi^{\top}q_i}\right)^{*} M(\xi) \left(\sum_{j=1}^{n}\beta_j e^{-i\xi^{\top}q_j}\right)\dd\xi.
\end{align*}
Here, $*$ represents the complex conjugate transpose. Define $F_\alpha(\xi) = \sum_{i=1}^{n}\alpha_i e^{-i\xi^{\top}q_i}$ and $F_\beta(\xi) = \sum_{j=1}^{n}\beta_j e^{-i\xi^{\top}q_j}$. Then we can write
\begin{equation}
    \label{eq:fourier_cometric_cpd}
    g_{s}^{-1}(\alpha,\beta) = \frac{1}{(2\pi)^{d/2}}\int_{\mathbb{R}^d} F_\alpha(\xi)^{*} M(\xi) F_\beta(\xi) \dd\xi.
\end{equation}

We will now consider the case $d=2$. We cannot use the naive inverse Fourier formula directly because it diverges and we need to use the regularized version \eqref{eq:ks-continuous-representation-d2}, but this does not change the final result since the contribution from the $-1$ term vanishes due to

\[
\begin{gathered}
\sum_{i,j=1}^{n}\alpha_i^{\top}
\int_{|\xi|\leq 1}
M(\xi)
\bigl(e^{i\xi^{\top}(q_i-q_j)}-1\bigr)
\beta_j\dd\xi
\\[6pt]
= 
\int_{\|\xi\|\leq 1}
\left(
  \sum_{i=1}^{n}\alpha_i e^{-i\xi^{\top}q_i}
\right)^{*}
M(\xi)
\left(
  \sum_{j=1}^{n}\beta_j e^{-i\xi^{\top}q_j}
\right)\dd\xi
\\[6pt] - \sum_{i=1}^{n}\alpha_i^{\top} \int_{|\xi|\leq 1} M(\xi) \sum_{j=1}^{n}\beta_j \dd\xi.
\end{gathered}
\]
Here, $\sum_{i=1}^{n}\alpha_i = 0$ and $\sum_{j=1}^{n}\beta_j = 0$ because of the centering assumption. Therefore, we obtain \eqref{eq:fourier_cometric_cpd}. We will use the Fourier representation \eqref{eq:fourier_cometric_cpd} of the cometric $g_{s}^{-1}(\alpha,\beta)$ in our subsequent analysis.

\subsubsection{Verifying the cometric comparison}
We will first prove the upper bound for the cometric comparison.

\begin{proposition}
    \label{prop:pd_cometric_upper_bound}
    For the CPD metric, if $s > d/2$, there exists a constant $B_2 > 0$ such that, for any covectors $\alpha$, 
    \begin{equation*}
        g_{s}^{-1}(\alpha,\alpha) \leq B_2 (1\lor r_0)^2 h^{-1}(\alpha, \alpha).
    \end{equation*}
\end{proposition}
\begin{proof}
    We note that, by the Fourier representation \eqref{eq:fourier_cometric_cpd}, we have
    \begin{equation*}
        g_{s}^{-1}(\alpha,\alpha) \leq \frac{1}{(2\pi)^{d/2}}\int_{\mathbb{R}^d} |F_\alpha(\xi)^{*} M(\xi) F_\alpha(\xi)| \dd\xi.
    \end{equation*}
    Here, $M(\xi)$ is the formal Fourier transform of $k_s(\xi)$ introduced in the previous proposition.  By Cauchy-Schwarz, we obtain
    \begin{align*}
        |F_\alpha(\xi)^{*} M(\xi) F_\alpha(\xi)| \leq \|F_\alpha(\xi)\|  \left\|M(\xi) F_\alpha(\xi)\right\| \leq \|F_\alpha(\xi)\|^2  \left\|M(\xi)\right\|_{op}.
    \end{align*}
    Here, $\|\cdot\|_{op}$ represents the operator norm of a matrix, defined as $\|A\|_{op} = \max_{\|v\|\leq 1} \|Av\|$.
    We note that $M(\xi)$ projects each vector onto two orthogonal components, so for any $v\in \mathbb{R}^d$ with $\|v\|\leq 1$, the norm of $M(\xi)v$ can be calculated as follows:
    \begin{align*}
        \|M(\xi)v\| &= \frac{\sqrt{\|P_{\perp}(\xi)v\|^2 + \frac{1}{4}\|P_{||}(\xi)v\|^2}}{(2\pi)^{d/2}(1+\sigma^2\|\xi\|^2)^{s}\|\xi\|^2} \\
        &\leq \frac{\sqrt{\|P_{\perp}(\xi)v\|^2 + \|P_{||}(\xi)v\|^2}}{(2\pi)^{d/2}(1+\sigma^2\|\xi\|^2)^{s}\|\xi\|^2} \\
        &= \frac{1}{(2\pi)^{d/2}(1+\sigma^2\|\xi\|^2)^{s}\|\xi\|^2} \sqrt{\|v\|^2} \notag\\
        &\leq \frac{1}{(2\pi)^{d/2}(1+\sigma^2\|\xi\|^2)^{s}\|\xi\|^2}.
    \end{align*}
    Therefore,  $\|M(\xi)\|_{op} =  \max_{\|v\|\leq 1} \|M(\xi)v\| \leq \frac{1}{(2\pi)^{d/2}(1+\sigma^2\|\xi\|^2)^{s}\|\xi\|^2}$. On the other hand, the norm $\|F_\alpha(\xi)\|^2$ can be bounded as follows by Cauchy-Schwarz:
    \begin{align*}
        \|F_\alpha(\xi)\|^2 &\leq  n\sum_{i=1}^{n}\|\alpha_i e^{-i\xi^{\top}q_i}\|^2 = n\sum_{i=1}^{n}\|\alpha_i\|^2 = nh^{-1}(\alpha, \alpha).
    \end{align*}
    We can use the fact that $\alpha_i$ are centered, i.e., $\sum_{i}\alpha_i=0$ to provide further bounds for $\|F_\alpha(\xi)\|^2$. By another application of Cauchy-Schwarz, we obtain
    \begin{align*}
        \|F_\alpha(\xi)\|^2 &= \left\|\sum_{i=1}^{n} \alpha_i e^{-i\xi^{\top}q_i}\right\|^2 \notag\\
        &= \left\|\sum_{i=1}^{n} \alpha_i \left(e^{-i\xi^{\top}q_i} - 1\right)\right\|^2 \notag\\
        &\leq \left(\sum_{i=1}^{n} \|\alpha_i\|^2\right) \left(\sum_{i=1}^{n} |e^{-i\xi^{\top}q_i} - 1|^2\right).
    \end{align*}
    By using the inequality $|e^{-i\xi^{\top}q_i} - 1| \leq |\xi^{\top}q_i| \leq \|\xi\| \|q_i\|$, we obtain
    \begin{align*}
        \sum_{i=1}^{n} |e^{-i\xi^{\top}q_i} - 1|^2 \leq \sum_{i=1}^{n} \|\xi\|^2 \|q_i\|^2 = \|\xi\|^2 \sum_{i=1}^{n} \|q_i\|^2 = \|\xi\|^2 r_0([q])^2.
    \end{align*}
    Therefore, we have
    \begin{align*}
        \|F_\alpha(\xi)\|^2 \leq h^{-1}(\alpha, \alpha) \|\xi\|^2 r_0([q])^2.
    \end{align*}
    Combining the two bounds, we obtain
    \begin{equation*}
        \|F_\alpha(\xi)\|^2 \leq \min\{n, \|\xi\|^2r_0([q])^2\}h^{-1}(\alpha, \alpha).
    \end{equation*}
    We now consider the following two cases.
    \paragraph{Case 1. $d \geq 3$.} Using our first bound, we get
    \begin{align*}
        g_s ^{-1}(\alpha,\alpha) &\leq  nh^{-1}(\alpha, \alpha) \int_{\mathbb{R}^d} \frac{1}{(2\pi)^{d}(1+\sigma^2\|\xi\|^2)^{s}\|\xi\|^2} \dd\xi.
    \end{align*}
    We will now consider the polar coordinates to evaluate the integral. Let $r = \|\xi\|$ and $\db\xi = r^{d-1} \dd r \dd\Omega$, where $\db\Omega$ represents the surface measure on the unit sphere $\mathbb{S}^{d-1}$ in $\mathbb{R}^d$. Then the integral becomes
    \begin{align*}
        &\int_{\mathbb{R}^d} \frac{1}{(2\pi)^{d}(1+\sigma^2\|\xi\|^2)^{s}\|\xi\|^2} \dd\xi \notag\\
        &\quad= \int_{0}^{\infty} \int_{\mathbb{S}^{d-1}} \frac{r^{d-1}}{(2\pi)^{d}(1+\sigma^2 r^2)^{s} r^2} \dd\Omega \dd r \\
        &\quad= \frac{\Omega(\mathbb{S}^{d-1})}{(2\pi)^{d}} \int_{0}^{\infty} \frac{r^{d-3}}{(1+\sigma^2 r^2)^{s}} \dd r.
    \end{align*}
    Around $r=0$, the integral behaves like $r^{d-3}$ which is integrable when $d-3 > -1$, i.e., $d > 2$. Around $r=\infty$, the integral behaves like $r^{d-3-2s}$  which is integrable when $d-3-2s < -1$, i.e., $2s > d-2$. Therefore, the integral converges for $d > 2$ and $2s > d-2$, which is covered by our assumption on $d$ and $s$. In particular, the comparison upper bound with $B_2=\frac{n\Omega(\mathbb{S}^{d-1})}{(2\pi)^d}\int_0^\infty \frac{r^{d-3}}{(1+\sigma^2r^2)^{s}}\dd r$ holds. 
    \paragraph{Case 2. $d=2$.} We note that the previous argument breaks down due to the behavior of the integral around $r=0$. We use our unified bound in this case to avoid singularity. By the unified bound, we have  
    \begin{align*}
    g_s^{-1}(\alpha,\alpha) \leq h^{-1}(\alpha, \alpha) \int_{\mathbb{R}^2} \frac{\min\{n, \|\xi\|^2 r_0([q])^2\}}{(2\pi)^{2}(1+\sigma^2\|\xi\|^2)^{s}\|\xi\|^2} \dd\xi.
    \end{align*}
    We again switch to polar coordinates, setting $r = \|\xi\|$ and $\db\xi = r \dd r \dd\Omega$, where $\db\Omega$ represents the surface measure on the unit circle $\mathbb{S}^{1}$ in $\mathbb{R}^2$. The integral becomes
    \begin{align*}
        &\int_{\mathbb{R}^2} \frac{\min\{n, \|\xi\|^2 r_0([q])^2\}}{(2\pi)^{2}(1+\sigma^2\|\xi\|^2)^{s}\|\xi\|^2} \dd\xi \\
        &\quad= \int_{0}^{\infty} \int_{0}^{2\pi} \frac{\min\{n, r^2 r_0([q])^2\}}{(2\pi)^{2}(1+\sigma^2 r^2)^{s} r^2} r \dd\theta \dd r \\
        &\quad= \frac{1}{2\pi} \int_{0}^{\infty} \frac{\min\{n, r^2 r_0([q])^2\}}{(1+\sigma^2 r^2)^{s} r} \dd r.
    \end{align*}
    We split the integral into two regions: $(0, 1]$ and $(1, \infty)$ to handle the behavior around $r=0$ and $r=\infty$ separately. On $(0,1)$, the integral is bounded by 
    \begin{equation*}
        \frac{1}{2\pi} \int_{0}^{1} \frac{r^2 r_0([q])^2}{(1+\sigma^2 r^2)^{s} r} \dd r \leq \frac{r_0([q])^2}{2\pi} \int_{0}^{1} r \dd r = \frac{r_0([q])^2}{4\pi}.
    \end{equation*}
    The $(1, \infty)$ integral is finite for $s > 0$ by the same argument as in the $d\geq 3$ case. Therefore, the integral converges for $d=2$ and $s>0$.
    Overall, we obtain the bound
    \begin{equation*}
        \begin{gathered}
        g_s^{-1}(\alpha, \alpha) \leq B_2(1\lor r_0)^2 h^{-1}(\alpha, \alpha),\\
        B_2 = \frac{1}{4\pi} + \frac{1}{2\pi}\int_{1}^{\infty} \frac{n}{(1+\sigma^2 r^2)^{s} r} \dd r.
        \end{gathered}
    \end{equation*}
    \end{proof}
    We will now proceed to show the lower bound for the CPD metric. The main idea is that the CPD metric is controlled from below by the metric induced by the Mat\'ern kernel, which is one of the scalar-kernel metrics covered by \cite{habermann2026stochasticcompletenesslandmarkspace}.
    \begin{proposition}
        \label{prop:cpd_lower_bound_matern}
        For every quotient covector $\alpha$, the CPD cometric is bounded from below by the scalar-kernel cometric associated with the Sobolev operator $(\mathrm{Id}-\sigma^2\Delta)^{s+1}$, namely,
        \begin{equation*}
            g_s^{-1}(\alpha, \alpha) \geq \frac{\sigma^2}{2} \sum_{i,j=1}^{n} \alpha_i^{\top} \alpha_j k_{\rm Matern}(q_i - q_j),
        \end{equation*}
        where $k_{\rm Matern}$ is the Green's function to the operator $(\mathrm{Id}-\sigma^2\Delta)^{s+1}$, so that $(\mathrm{Id}-\sigma^2\Delta)^{s+1}k_{\rm Matern}=\delta_0$. With our Fourier convention, we have $\widehat{\delta_0}=(2\pi)^{-d/2}$, and hence
        \begin{equation*}
            \widehat{k}_{\rm Matern}(\xi) = \frac{1}{(2\pi)^{d/2}}(1+\sigma^2\|\xi\|^2)^{-(s+1)}.
        \end{equation*}
        Taking the inverse Fourier transform gives
        \begin{equation*}
            k_{\rm Matern}(x) = \frac{1}{(2\pi)^d}\int_{\mathbb{R}^d} (1+\sigma^2\|\xi\|^2)^{-(s+1)} e^{i \xi ^\top x} \dd\xi.
        \end{equation*}
    \end{proposition}
    \begin{proof}
        Remember that the Fourier representation of the cometric $g_s^{-1}$ is defined by the following formal Fourier transform of $k_s$:
        \begin{align*}
        M(\xi) &= \frac{1}{(2\pi)^{d/2} (1+\sigma^2 \|\xi\|^2)^{s}\|\xi\|^2} \left(P_{\perp}(\xi) + \frac{1}{2} P_{||}(\xi)\right).
        \end{align*}
        Now consider the order $\geq$ of matrices where $A\geq B$ if and only if $A-B$ is positive semidefinite. Under this order, for $\xi\neq 0$, we have
        \begin{align*}
            M(\xi) & \geq \frac{1}{2(2\pi)^{d/2} (1+\sigma^2 \|\xi\|^2)^{s}\|\xi\|^2} \left(P_{\perp}(\xi) +  P_{||}(\xi)\right) \\ &\geq \frac{\sigma^2}{2(2\pi)^{d/2} (1+\sigma^2 \|\xi\|^2)^{s}\sigma^2\|\xi\|^2} I_d \\ &\geq \frac{\sigma^2}{2(2\pi)^{d/2} (1+\sigma^2 \|\xi\|^2)^{s}(1+\sigma^2\|\xi\|^2)} I_d \\ &= \frac{\sigma^2}{2(2\pi)^{d/2}}(1+\sigma^2\|\xi\|^2)^{-(s+1)} I_d.
        \end{align*}
        Therefore, by multiplying on the left by $F_\alpha(\xi)^*$ and on the right by $F_\alpha(\xi)$, and integrating with the prefactor $(2\pi)^{-d/2}$ from \eqref{eq:fourier_cometric_cpd}, we obtain
        \begin{align*}
            g_s^{-1}(\alpha, \alpha) \geq \frac{\sigma^2}{2}\sum_{i,j=1}^{n} \alpha_i^{\top} \alpha_j \left(\frac{1}{(2\pi)^d}\int_{\mathbb{R}^d} (1+\sigma^2\|\xi\|^2)^{-(s+1)} e^{i \xi^\top (q_i - q_j)} \dd\xi \right).
        \end{align*}
        Substituting the definition of $k_{\rm Matern}$, we obtain the proposition. 
    \end{proof}

    As discussed in the previous work \cite{habermann2026stochasticcompletenesslandmarkspace}, the scalar kernel $k_{\rm Matern}$ is one of the kernels satisfying Assumption \ref{asm:kernel hypothesis}, so by using the lower bound for the scalar kernel metric, Proposition \ref{prop:scalar-verification}, we obtain the desired lower bound for the CPD metric.
    \begin{corollary}
        There exist constants $B_1>0$ and $B_3>0$ so that for any covectors $\alpha$, we satisfy the lower bound of Assumption \ref{assumption:main_metric}, that is,
        \begin{equation*}
            B_1 (1\land \delta([q]))^{B_3} h^{-1}(\alpha, \alpha) \leq g_s^{-1}(\alpha, \alpha).
        \end{equation*}
    \end{corollary}

    \subsubsection{Verifying the collision speed control}
    
    We will now show the collision speed control condition in Assumption \ref{assumption:main_metric} holds.
    \begin{theorem} For the CPD metric, if $s > d/2$, there exists a constant $B_2 > 0$ and $r_*>0$ such that the collision speed control condition in Assumption \ref{assumption:main_metric} holds, that is, if $0<\rho_{ij}<r_*$, then
        \begin{equation*}
            |\nabla_{g_s} \rho_{ij}|_{g_s} \leq B_2 \rho_{ij}\sqrt{\log{\left(\frac{e r_*}{\rho_{ij}}\right)}}.
        \end{equation*}
    \end{theorem}
    \begin{proof}
    We will calculate $g^{-1}_{s}(\db\rho_{ij}, \db\rho_{ij})$ for $[q]\in \Sreg$. By the same argument as in Proposition \ref{prop:pd-speed-control}, we have
    \begin{equation*}
        \db\rho_{ij} = (\nabla_{q_1} \rho_{ij}, \nabla_{q_2} \rho_{ij}, \dots, \nabla_{q_n} \rho_{ij}).
    \end{equation*}
    Each component is calculated as
    \begin{equation*}
        \nabla_{q_k} \rho_{ij} =
        \begin{cases}
            \frac{q_k - q_j}{\rho_{ij}}, & \text{if } k=i,\\
            \frac{q_k - q_i}{\rho_{ij}}, & \text{if } k=j,\\
            0, & \text{otherwise}.
        \end{cases}
    \end{equation*}
    We will now use the Fourier representation \eqref{eq:fourier_cometric_cpd} to calculate the cometric. We have
    \begin{align*}
    g^{-1}_{s}(\db\rho_{ij}, \db\rho_{ij}) = \frac{1}{(2\pi)^{d/2}} \int_{\mathbb{R}^d}F_{\db\rho_{ij}}(\xi)^* \hat{k}_s(\xi) F_{\db\rho_{ij}}(\xi) \dd\xi.
    \end{align*}
    We will calculate the bound of $\|F_{\db\rho_{ij}}(\xi)\|$ explicitly. By definition,
    \begin{align*}
        \norm{F_{\db\rho_{ij}}(\xi)} &= \norm{e^{-i  \xi^T q_i } \frac{q_i - q_j}{\rho_{ij}} + e^{-i \xi^T q_j} \frac{q_j - q_i}{\rho_{ij}}} \\ &= |1 - e^{i\xi^{\top} (q_j - q_i)}| \frac{\norm{q_i - q_j}}{\rho_{ij}} \\ &= |1 - e^{i\xi^{\top} (q_j - q_i)}| \leq  \|\xi \| \|q_i-q_j\| = \|\xi\| \rho_{ij},
    \end{align*}
    where we have used the fact that $|1 - e^{i \theta}| \leq |\theta|$ for any real number $\theta$.
    We combine this bound with the operator norm of $\hat{k}_s(\xi)$ calculated in Proposition \ref{prop:pd_cometric_upper_bound} to obtain an upper bound for $g^{-1}_{s}(\db\rho_{ij}, \db\rho_{ij})$:
    \begin{align*}
        g^{-1}_{s}(\db\rho_{ij}, \db\rho_{ij}) &\leq \frac{1}{(2\pi)^{d/2}} \int_{\mathbb{R}^d} \|\xi\|^2 \rho_{ij}^2 \|\hat{k}_s(\xi)\|_{\text{op}} \dd\xi \\ 
        &\leq \frac{\rho_{ij}^2}{(2\pi)^{d/2}}\int_{\mathbb{R}^d} \frac{\|\xi\|^2}{(2\pi)^{d/2}(1+\sigma^2\|\xi\|^2)^{s}\|\xi\|^2}\dd\xi \\ 
        &= \frac{\rho_{ij}^2}{(2\pi)^{d}}\int_{\mathbb{R}^d} \frac{1}{(1+\sigma^2\|\xi\|^2)^{s}}\dd\xi.
    \end{align*}
    We will now evaluate the integral by using polar coordinates. Let $\|\xi\| = r$, then $\db\xi = r^{d-1} \dd r \dd\Omega$, where $\db\Omega$ is again the surface measure of the $d-1$-dimensional unit sphere. The integral becomes
    \begin{align*}
        &\frac{\rho_{ij}^2}{(2\pi)^{d}} \int_0^\infty \int_{\mathbb{S}^{d-1}} \frac{r^{d-1}}{(1+\sigma^2 r^2)^{s}} \dd\Omega \dd r \notag\\
        &\quad= \frac{\rho_{ij}^2}{(2\pi)^{d}} \int_0^\infty \frac{r^{d-1}}{(1+\sigma^2 r^2)^{s}} \dd r \int_{\mathbb{S}^{d-1}} \dd\Omega \\
        &\quad= \frac{\rho_{ij}^2 \Omega(\mathbb{S}^{d-1})}{(2\pi)^{d}} \int_0^\infty \frac{r^{d-1}}{(1+\sigma^2 r^2)^{s}} \dd r.
    \end{align*}
    Now, the remaining $r$-integrand is of order $r^{d-1}$ around $r=0$ and of order $r^{d-1-2s} = r^{d-2s-1}$ around $r=\infty$. Therefore, the integral converges if $d-2s-1 < -1$, i.e., $s > d/2$, which is covered by our assumption. Therefore, we have shown the speed collision control in Assumption \ref{assumption:main_metric} except for the logarithmic term. To create the logarithmic term, set $r_*=1$ and assume that $0<\rho_{ij}<1$. We then have that $\log{\left(\frac{e}{\rho_{ij}}\right)} > 1$, so we have 
    \begin{align*}
        g^{-1}_{s}(\db\rho_{ij}, \db\rho_{ij}) &\leq (B_2)^2 \rho_{ij}^2 \log{\left(\frac{e r_*}{\rho_{ij}}\right)}, \notag\\
        (B_2)^2 &= \frac{\Omega(\mathbb{S}^{d-1})}{(2\pi)^{d}} \int_0^\infty \frac{r^{d-1}}{(1+\sigma^2 r^2)^{s}} \dd r.
    \end{align*}
    Taking the square root on both sides, we obtain the desired bound.
    \end{proof}

\section{Stochastic completeness after scale normalization}\label{sec:scale}
In this section, we will discuss how the argument for the unnormalized configuration model can be adapted to the normalized configuration model. Under the regular stratum of the quotient manifold $\Sreg$, the normalized submanifold $\Sregone$ is the set of configurations with unit norm, i.e., $r_0([q]) = 1$.  We note that $\Sregone$ is connected under $d\geq 2, n\geq d+2$ since the normalization map $N([q]) = \frac{[q]}{r_0([q])}$ is continuous and surjective onto $\Sregone$, so the connectedness of $\Sreg$ implies the connectedness of $\Sregone$. We recall that the connectedness matters to apply Masamune's theory on stochastic completeness.
\subsection{The assumptions descend to $\Sregone$}
We start by remarking that the assumptions on the cometric in Assumption \ref{assumption:main_metric} survive the restriction to the normalized submanifold $\Sregone$. Denote the given quotient metric on $\Sreg$ by $g$ and the induced metric on $\Sregone$ by $g_1$.

We first consider the comparison assumption. We recall that the metric in the submanifold is defined by the restriction of the ambient metric to the tangent space of the submanifold. Therefore, we invert the cometric bounds to metric bounds, and then restrict these bounds to the tangent space of the normalized submanifold $\Sregone$. Inverting again, we obtain the cometric bounds for the normalized submanifold.

Similarly, the collision speed assumption in Assumption \ref{assumption:main_metric} also descends to the normalized submanifold $\Sregone$ because the gradient on the submanifold is the orthogonal projection of the ambient gradient onto the tangent space of the submanifold, so
\begin{equation}
    \|\nabla_{g_1} \rho_{ij}\|_{g_1} \leq \|\nabla_{g} \rho_{ij}\|_{g} \leq B_2 \rho_{ij} \sqrt{\log{\left(\frac{e r_*}{\rho_{ij}}\right)}},
\end{equation}
whenever $\rho_{ij} < r_*$ for some constant $r_* > 0, B_2>0$.

The consequence is that the Lemma \ref{lemma:bounds_on_metric_balls} and Corollary \ref{cor:comparison_g_metric_balls} hold for the normalized submanifold $\Sregone$, though the $r_0$ bound is not necessary anymore.  The results can be used to prove the boundary characterization as in the unnormalized case.

\subsection{Boundary characterization after scale normalization}
We will now show the boundary characterization for the normalized submanifold $\Sregone$, which requires a different but shorter argument compared to the unnormalized case.

\begin{theorem}
    Consider the metric completion of the normalized submanifold $\Sregone$ with respect to the induced distance $d_{g_1}$ by the metric $g_1$. The functions $\delta$ and $\sigma_{d}$ on $\Sregone$ both extend continuously to its completion. Moreover, any point $[q]\in \partial \Sregone$ satisfies $\delta([q])>0$ and $\sigma_{d}([q]) = 0$, where $\delta$ and $\sigma_d$ denote the continuous extensions.
\end{theorem}
\begin{proof}
    Similar to the unnormalized case, we first prove that $\delta$ and $\sigma_d$ are Cauchy continuous, but this is true because $d_{g} \leq d_{g_1}$ so a Cauchy sequence in $g_1$ is Cauchy in $d_g$, and the previous Cauchy continuity on $g$ translates to our case.

    We will now show that any boundary point $[q]\in \partial \Sregone$ satisfies $\delta([q])>0$ and $\sigma_{d}([q]) = 0$. The former is true by the same continuity argument as in the unnormalized case. For the latter, we do essentially the same argument, but we will record it for completeness. Assume for the sake of contradiction that $\sigma_d([q]) > 0$. In that case, we show that any point in the boundary of the completion can be realized as a configuration rather than an abstract completion point. 
    
    To show this, notice that, by definition of the metric boundary, there exists a $g_1$-Cauchy sequence $[q_n]\in \Sregone$ that converges to $[q]$ in the completion. Take a representative $q_n$ from each quotient point $[q_n]$ and form a sequence in $\Landz$ with constraint $r_0(q_n)=1$. In the ambient Euclidean sphere, there exists a convergent subsequence $q_{n_k}$ by compactness. Now denote the limit by $\bar{q}$. We will show that the quotient point $[\bar{q}]$ is a limit of $[q_{n_k}]$ in the completion, and by the uniqueness of limits $[\bar{q}]=[q]$. In fact, by continuity, we can take a Euclidean ball containing $\bar{q}$ and $q_{n_k}$ so that $\delta(q) \geq \frac{1}{2}\delta(\bar{q})$ and $\sigma_d>0$ for $q$ inside the ball for large enough $k$. Connect $\bar{q}$ and $q_{n_k}$ by the straight line $\gamma(t)$ and normalize them to obtain $z(t) = \gamma(t)/\|\gamma(t)\|_F$. If $k$ is large enough, the Frobenius norm of the difference $\|\gamma(t) - \bar{q}\|_F$ is small for all $t$, so $\|\gamma(t)\|_F$ is bounded both from above and below. So, $\delta(z(t)) = \frac{\delta(\gamma(t))}{\|\gamma(t)\|_F}$ is uniformly bounded from below. Similarly, $\sigma_d(z(t))>0$ because the singular values scale linearly. Now, the length of $z(t)$ satisfies, by the analogue of the comparison assumption in Assumption \ref{assumption:main_metric},
    \begin{equation}
        d_{g_1}([q_{n_k}], [\bar{q}])\leq L_{g_1}(z(t)) \leq E_3 L_{h_1}(z) \leq E_3 L_{\mathbb{S}}(z) =  E_3 d_{\mathbb{S}}(q_{n_k}, \bar{q}) \to 0, 
    \end{equation}
    where $E_3$ is defined analogously to the unnormalized case, $L_{\mathbb{S}}$ is the length on the sphere, and $d_{\mathbb{S}}$ is the distance on the sphere. Note that the last equality comes from the fact that normalizing the straight line traces a geodesic on the sphere. Therefore, $[q_{n_k}]$ converges to $[\bar{q}]$, so by the uniqueness $[q] = [\bar{q}]$, a contradiction since $[\bar{q}]\in \Sregone$ by the assumption that $\sigma_d([q]) = \sigma_d(\bar{q}) > 0$.
\end{proof}

\subsection{The SVD formula on the normalized submanifold}
We will now prove the SVD formula for the normalized submanifold, which follows immediately from the unnormalized case.

\begin{proposition}
    \label{prop:svd-formula-normalized}
    There is a constant $G_1>0$, depending only on $n$ and $d$, such that for every $\epsilon>0$ and every nonnegative nonincreasing function $w\colon (0,\infty) \to [0,\infty)$, 
    \begin{equation}
        \int_{0 < \sigma_d < \epsilon} w(\sigma_d) \dd\mu_{h_1} \leq (n-d)G_1\int_{0}^{\epsilon}w(s) s^{n-d-1}\dd s.
    \end{equation}
    In particular, taking $w\equiv 1$ gives the volume bound
    \begin{equation}
        \mu_{h_1}(\{[q]\in \Sregone : 0 < \sigma_d([q]) < \epsilon \}) \leq G_1\epsilon^{n-d}.
    \end{equation}
\end{proposition}
Before the proof, we remark that the assumption on $w$ is stronger here. In the unnormalized case, we only assumed nonnegativity on $w$, but now we assume that it is nonincreasing. As we demonstrate later, this does not cause any issues in downstream use.

\begin{proof}
    We use the polar decomposition $[q] = [r\bar{q}]$ where $q\in \Sreg$, $\bar{q}\in \Sregone$ and $r = \|q\|_F$. The Euclidean quotient measures decompose as $\db\mu_h = r^{D-1}\dd r \dd\mu_{h_1}$ where $D$ is the dimension of $\Sreg$. Therefore, by Proposition \ref{ref:singular_value_bound} with $R=1$, we obtain
    \begin{align*}
    \int_{\{\sigma_d < \epsilon\}} \int_{0}^{1} w(\sigma_d(r\bar{q})) r^{D-1}\dd r \dd\mu_{h_1}
    &\leq (n-d)G_1(1) \int_{0}^{\epsilon} w(s)s^{n-d-1}\dd s.
    \end{align*}
    We omit the argument $1$ in $G_1(1)$ and simply write $G_1$. We note that the singular values scale linearly, so $\sigma_d(r\bar{q}) = r\sigma_d(\bar{q})$, and by the nonincreasing assumption, we have $w(r\sigma_d(\bar{q})) \geq  w(\sigma_d(\bar{q}))$ for $0<r<1$. Therefore,
    \begin{align*}
        \int_{\{\sigma_d < \epsilon\}} \int_{0}^{1} w(\sigma_d(r\bar{q})) r^{D-1}\dd r \dd\mu_{h_1} \geq \int_{\{\sigma_d < \epsilon\}} \int_{0}^{1} w(\sigma_d(\bar{q})) r^{D-1}\dd r \dd\mu_{h_1}  \\ \geq \left(\int_{0}^{1}r^{D-1}\dd r\right)\int_{\{\sigma_d < \epsilon\}} w(\sigma_d(\bar{q}))  \dd\mu_{h_1},
    \end{align*}
    and the result follows by redefining $G_1$ as $G_1/\int_{0}^{1}r^{D-1}dr = DG_1$.
\end{proof}

\subsection{Almost polar boundary and volume growth condition}
We will now consider the proof for the two main components for the stochastic completeness: the almost polar boundary and the volume growth condition.

We will first prove the fact that the metric boundary has capacity zero. Because of the tools we have established so far, the proof can be written almost verbatim with one fix. We recall that the SVD formula for the normalized case, namely, Proposition \ref{prop:svd-formula-normalized}, only applies to nonincreasing functions. However, the proof of the unnormalized case, Theorem \ref{thm:boundary-almost-polar}, applies Proposition \ref{ref:singular_value_bound} to
\begin{equation*}
w=|\psi_{\epsilon}|^2  + |\psi_{\epsilon}'|^2,\quad  
        \psi_{\epsilon}(s) = \begin{cases}
            1, & s \leq \epsilon, \\
            2 - \frac{s}{\epsilon}, & \epsilon < s < 2\epsilon, \\
            0, & s \geq 2\epsilon,
        \end{cases}
\end{equation*}
which is not nonincreasing because $\psi' = (-1/\epsilon) 1_{\{\epsilon < s < 2\epsilon\}}$ where $1_S$ is the indicator for the set $S$. This can be resolved by taking the nonincreasing majorant $(1 + 1/\epsilon^2)1_{\{0< s<2\epsilon\}}$ as $w$, which is what we do regardless to calculate the integral in the unnormalized case. Therefore, after this fix, the proof of almost polarity follows after replacing metrics and dimensions with the corresponding versions.

We will now move on to the proof of the fact that the volume of the metric ball $v(r)$ of radius $r$ satisfies $\int^{\infty} \frac{r}{\log{v(r)}}\dd r =\infty$ or $\int^{\infty} \frac{r}{v(r)}\dd r =\infty$. The main input to the proof in the unnormalized case was Proposition \ref{prop:volume_growth_metric_balls}, which then relies on Lemma \ref{lemma:bounds_on_metric_balls}, Corollary \ref{cor:comparison_g_metric_balls} and the SVD formula in Proposition \ref{ref:singular_value_bound}. As we have discussed, these results can be proved for the normalized version with no or slight modification, so the proof works after replacing metrics and dimensions. 

By the above two results, we have stochastic completeness for the scale normalized submanifold.

\bibliographystyle{amsplain}
\bibliography{references}

\end{document}